\documentclass[10pt]{amsart}%
\usepackage[T1]{fontenc}
\usepackage{amsmath,amssymb}
\usepackage{fullpage}
\usepackage{amsthm}
\usepackage{multirow}
\usepackage{color}
\usepackage{tikz}
\usepackage{url}
\usepackage{young}%
\usepackage{pifont}
\usepackage{amsmath}
\usepackage{graphicx}
  \usepackage[all]{xy}
\usepackage{amsfonts}
\usepackage{amssymb}

\providecommand{\U}[1]{\protect\rule{.1in}{.1in}}
\providecommand{\U}[1]{\protect\rule{.1in}{.1in}}
\providecommand{\U}[1]{\protect\rule{.1in}{.1in}}
\providecommand{\U}[1]{\protect\rule{.1in}{.1in}}
\providecommand{\U}[1]{\protect\rule{.1in}{.1in}}
\newcommand{\ulambda}{{\boldsymbol{\lambda}}}

\newcommand{\umu}{{\boldsymbol{\mu}}}
\newcommand{\unu}{{\boldsymbol{\nu}}}

\newcommand{\uemptyset }{{\boldsymbol{\emptyset}}}
\newtheorem{Th}{Theorem}[section]

\newtheorem{Cor}[Th]{Corollary}
\newtheorem{Prop}[Th]{Proposition}

\theoremstyle{remark}
\newtheorem{Rem}[Th]{Remark}{\rmfamily}
\theoremstyle{definition}
\newtheorem{Def}[Th]{Definition}{\rmfamily}
\newtheorem{exa}[Th]{Example}{\rmfamily}

\newcommand{\Core}{\operatorname{Core}}

\newcommand\blfootnote[1]{%
  \begingroup
  \renewcommand\thefootnote{}\footnote{#1}%
  \addtocounter{footnote}{-1}%
  \endgroup
}
\begin{document}

\title{Cores of Ariki-Koike algebras}
\author{Nicolas Jacon}
\email{nicolas.jacon@univ-reims.fr} 
\author{C\'edric Lecouvey}
\email{cedric.lecouvey@lmpt.univ-tours.fr} 
\address{N.J.: Universit\'{e} de Reims Champagne-Ardennes, UFR Sciences exactes et
naturelles Laboratoire de Math\'{e}matiques UMR CNRS 9008 Moulin de la Housse BP
1039 51100 REIMS. \email{nicolas.jacon@univ-reims.fr}\\
C. L.: Institut Denis Poisson UMR CNRS 7013 Facult\'{e} des Sciences et
Techniques, Universit\'{e} Fran\c{c}ois Rabelais Parc de Grandmont, cedric.lecouvey@lmpt.univ-tours.fr}

\maketitle
\date{}
\blfootnote{\textup{2010} \textit{Mathematics Subject Classification}: \textup{20C08,20C20,05E15}} 
\begin{abstract}
We study a natural generalization of the notion of cores for $l$-partitions attached with a multicharge ${\bf s}\in \mathbb{Z}^l$:  the  $(e,{\bf s})$-cores.  We rely them both to the combinatorics and the notion of weight  defined by Fayers. Next we study applications in the context of the block theory for Ariki-Koike algebras.  
\end{abstract}

\section{Introduction}

Let $\mathbb{F}$ be a field of characteristic $p\geq 0 $.  Let $l$ and $n$ be positive integers and ${\bf s}:=(s_1,\ldots,s_l)\in \mathbb{Z}^l$. Fix $\eta\in\mathbb{F}^*$. 
  The Ariki-Koike algebra $\mathbb{F}\mathcal{H}_n^{{\bf s}}(\eta)$ 
  associated with this datum is 
the unital associative $\mathbb{F}$-algebra with a presentation by:

\begin{itemize}
\item  generators: $T_0$, $T_1$,..., $T_{n-1}$,

\item  relations: 
\begin{align*}
& T_0 T_1 T_0 T_1=T_1 T_0 T_1 T_0, \\
& T_iT_{i+1}T_i=T_{i+1}T_i T_{i+1}\ (i=1,...,n-2), \\
& T_i T_j =T_j T_i\ (|j-i|>1), \\
&(T_0-\eta^{s_1})(T_0-\eta^{s_2})...(T_0- \eta^{s_{l}}) = 0, \\
&(T_i-\eta)(T_i+1) = 0\ (i=1,...,n-1).
\end{align*}
\end{itemize}
Let $e\geq 2$ be minimal such that $1+\eta+\ldots +\eta^{e-1}=0$ in $\mathbb{F}$ so that $e\in \{2,3,\ldots \}\cup\{\infty\}$. If $p=0$, $e$ is the order of $\eta$ as a root of unity. If $p>0$, we have 
 $e=p$ if and only if $\eta=1$. 

 The Ariki-Koike algebra, also called Hecke algebra of the complex reflection group $G(l,1,n)$, has been intensively studied during the last past  decades. It is in relation with various important objects (e.g. rational Cherednik algebras, quantum groups, finite reductive groups etc.) and has a deep representation theory.  Recently, the interest on these algebras have even grew up thanks to the introduction of the quiver Hecke algebras which has strengthened their  relations with the theory of quantum groups and  has allowed the study of their graded representation theory. 

When $l=1$, the Ariki-Koike algebra is nothing but the Hecke algebra of the symmetric group (when $\eta=1$, it is isomorphic to the group algebra of the symmetric group over $\mathbb{F}$).
 In the case where $e$ is a prime number, the representation theory of this algebra presents strong analogies with the modular representation theory of the symmetric groups (in characteristic $e$): both structures  admit a class of remarkable finite dimensional modules indexed by the set of partitions of $n$: the Specht modules. The simple modules are indexed by the set of  $e$-regular partitions and the decomposition matrices, which control their representation theories, can be connected  using an adjustment matrix.  Using these decomposition matrices, one can obtain a natural partition of  the set of Specht modules into smaller subsets called blocks. To each block, one can also associate another notion: the weight.  Roughly speaking, this positive integer  measures how ``complicated'' this block is.  
  Remarkably, one can describe the blocks and the weights quite easily using well known combinatorial notions. 
   In particular, the most ``simple'' blocks, the blocks with weight $0$,  can be described explicitly: they are singletons  consisting of a unique Specht module labeled by an $e$-core partition. Moreover, any block with a given weight $w$ may be obtained from a simple block by adding $w$ times  $e$-hooks to its associated $e$-core partition.  Importantly, all these properties still make sense  when $e$ is an arbitrary positive integer (strictly greater than $1$) on the side of the Hecke algebra.

 When $l>1$, one can also define analogues of Specht modules.  They are  now indexed by the set of $l$-partitions of $n$. The simple modules are then  naturally indexed by 
   certains generalizations of $e$-regular partitions which depend on ${\bf s}$: the Uglov $l$-partitions. A notion of weight has also been provided  by Fayers in \cite{Fa} which generalizes the case $l=1$. Thanks to this definition, many properties known  in the case $l=1$ have   been  extended  to the general  case $l\in \mathbb{N}$. In particular in \cite{LM}, Lyle and Mathas have given a necessary and sufficient condition for two Specht modules 
    for being in the same block. However,   the generalization of $e$-core partitions and a generalization of the above process  of adding $e$-hooks 
       were missing in this picture  (even if, as explained in \S \ref{expl}, a non explicit definition of core multipartitions has  been given by Fayers in \cite{Fa3}).

       The aim of this paper is to study in details the  $(e,{\bf s})$-core $l$-partitions, as introduced in a recent paper by the authors \cite{JL}. We show that this notion 
        gives the right generalization of the $e$-core partitions: they correspond to the elements with weight $0$ (with respect to Fayers definition of weight), and all $l$-partitions with a given weight may be obtained from them by adding analogues of $e$-hooks.  As a consequence, we obtain a direct and simple generalization of what happen in the case $l=1$. The only difference with this latter case is that, in our definition,  the core of an $l$-partition associated with a multicharge  is also a multipartition associated with a multicharge but, this  last multicharge may be different from the initial one.  To do this, 
 the strategy is to  show that essentially all the theory can be derived from the case $l=1$ by introducing a weight-preserving map, defined by Uglov,  from the set of 
 $(e,{\bf s})$-core $l$-partitions to the set of $e$-core partitions.        
 
         The paper will be organized as follows. We first recall the definition of our main object of study: the $(e,{\bf s})$-cores and provide some of their combinatorial properties. The third part studies the weights of the $l$-partitions as defined by Fayers. We show how this notion can be interpreted in the theory of Fock spaces and computed via a combinatorial procedure detailed in our last section. This section will also explore some consequences of our results and will explain how our approach can  simplify the block theory for Ariki-Koike algebras.

\section{Generalized cores and abaci}

In this section, after recalling certain classical  combinatorial definitions regarding the partitions,  we introduce the notion of $(e,{\bf s})$-core multipartition. Then we use abaci to associate to each  $(e,{\bf s})$-core a certain core partition. This section will be purely combinatorial.  

\subsection{Partitions and multipartitions}
\label{subset_multi}
A {\it partition} is a nonincreasing
sequence $\lambda=(\lambda_{1},\cdots,\lambda_{m})$ of nonnegative
integers. One can assume this sequence is infinite by adding parts equal to
zero. The rank  of the partition is by  definition the number $|\lambda|=\sum_{1\leq i\leq m} \lambda_i$. 
 We say that $\lambda$ is a partition of $n$. By convention, the unique partition of $0$ is the empty partition $\emptyset$. 

More generally, for $l\in \mathbb{Z}_{>0}$, an {\it $l$-partition} $\ulambda$ of $n$ is a sequence of $l$ partitions $(\lambda^1,\ldots,\lambda^l)$ 
 such that the sum  of the ranks of the $\lambda^j$ is $n$. The number $n$ is then called  the rank of $\ulambda$ and it is denoted by $|\ulambda|$.  The set of $l$-partitions is denoted by $\Pi^l$. 
 The {\it nodes}  of $\ulambda$ are by definition the elements of   the {\it Young diagram} of $\ulambda$:
$$[\ulambda]:=\{ (a,b,c)  \ | \ a\geq 1,\ c\in \{1,\ldots,l\},\ 1\leq b\leq \lambda_a^c\} \subset \mathbb{Z}_{>0}\times 
  \mathbb{Z}_{>0} \times \{1,\ldots,l\}.$$
(in the case of partition, the third coordinate which is always equal to $1$ will be sometimes omitted.)
 Each $l$-partition will be identified with its Young diagram. We say that 
  a node of $\ulambda$ is {\it removable} when one can remove it from the Young diagram of $\ulambda$ and still get  the Young diagram of an $l$-partition $\umu$. In this case, this node is called an {\it  addable  node} for $\umu$. 
       
\begin{exa}
For $l=2$, the $2$-partition $((4),(2,1))$ of $7$ is identified with its Young diagram:
  $$
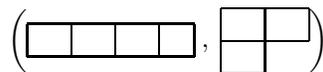
\left(
\begin{array}{|c|c|c|c|}
  \hline
  \  \ &\ \  &\ \ &\     \\
  \hline
\end{array}\;,\;
\begin{array}{|c|c|}
  \hline
  \ \  &\ \      \\
  \cline{1-2}
 \ \   \\
 \cline{1-1}
\end{array}
\right)$$

\end{exa}
Now let us come back to the case $l=1$ (we refer to \cite{mac} for details).  Let $e\in \mathbb{N}_{>1}$. A {\it rim $e$-hook} (or simply an $e$-hook) of a partition $\lambda$   is a connected subset of the rim of $\lambda$ with exactly $e$ nodes and which can be removed from $\lambda$ to obtain another partition 
$\mu$ as in the following example. 
\begin{exa}
Let $\lambda:=(5,4,2,1,1)$ and $e=3$. The Young diagram of $\lambda$ is:
  $$
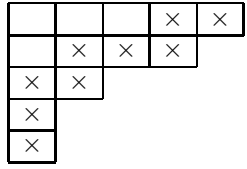
\begin{array}{|c|c|c|c|c|}
  \hline
  \  \ &\ \  &\ \ & \times &  \times  \\
  \cline{1-5}
  \  \ &\times &\times & \times \\
  \cline{1-4}
    \times &\times   \\
      \cline{1-2}
         \times \\
  \cline{1-1}
         \times  \\
  \cline{1-1}
\end{array}$$
Starting from above,  we can successively remove  three rim $3$-hooks (indicated with the symbol $\times$ above)

\end{exa}

By definition an {\it $e$-core} is a partition which does not admit any rim $e$-hook. The set of $e$-core partitions is denoted by $\mathfrak{C} (e)$. If $\lambda$ is an arbitrary partition, the {\it $e$-weight} $\omega_e (\lambda)$ of $\lambda$ is the number 
 of consecutive   $e$-hooks which can be removed from $\lambda$ before obtaining an $e$-core, which is then denoted by $\Core_e (\lambda)$. These notions are well-defined since both $\omega_e (\lambda)$ and $\Core_e (\lambda)$ do not depend on the order in which the rim $e$-hooks are removed from $\lambda$.

\begin{exa}\label{excore}
Keeping the above example, we obtain $\omega_3 (\lambda)=3$ and $\Core_3 (\lambda)=(3,1)$.
\end{exa}

Let ${\bf s}=(s_1,\ldots,s_l)\in \mathbb{Z}^l$. This is called a {\it multicharge} (a charge if $l=1$). 
For an $l$-partition $\ulambda=(\lambda^1,\ldots,\lambda^l)$, one can associate to each node $(a,b,c)$  of the Young diagram its {\it residue}
 $b-a+s_c+e\mathbb{Z} \in \mathbb{Z}/e \mathbb{Z}$. The set of residues will be identified with $\{0,\ldots,e-1\}$. 
  If $i\in \mathbb{Z}/e\mathbb{Z}$, we denote by $c^{e,{\bf s}}_{i} (\ulambda)$ the number of nodes with residue $i$ in the $l$-partition.  We moreover denote 
   $\mathcal{C}_{e,{\bf s}} (\ulambda):=(c_0(\ulambda),\ldots,c_{e-1} (\ulambda))$. 
 
\begin{exa}
For $l=2$, ${\bf s}=(0,1)$ and $e=3$ the residues of the nodes of the  $2$-partition $((4),(2,1))$ of $7$ are given as follows:
  $$
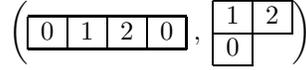
\left(
\begin{array}{|c|c|c|c|}
  \hline
  0& 1  &2 &0   \\
  \hline
\end{array}\;,\;
\begin{array}{|c|c|}
  \hline
  1  &2      \\
  \cline{1-2}
 0   \\
 \cline{1-1}
\end{array}
\right)$$
Here we have  $\mathcal{C}_{e,{\bf s}} ((4),(2,1))=(3,2,2)$.

\end{exa}

\subsection{Abaci}\label{ab}

The notion of abacus is convenient for reading the weight of a partition and obtaining its $e$-core. Let $s\in \mathbb{Z}$. 
An {\it abacus}   is a subset $\mathcal{A}$ of $\mathbb{Z}$
  such that $-i\in \mathcal{A}$ and $i\notin \mathcal{A}$ for all $i$ large enough.
    In a less formal way, each $i\in \mathcal{A}$ corresponds to the position of a black bead
on the horizontal abacus  which is full of  black beads on the left and empty on the right. 
 One can associate to $\lambda$ and $s\in \mathbb{Z}$ an abacus $L_s (\lambda)$ 
  such that $k\in \mathcal{A}$ if and only if there exists $j\in \mathbb{N}$ such that $k=\lambda_j-j+s$ (Note that 
  $\lambda$  is assumed to have an
infinite number of zero parts). Given an abacus $L$, one can easily find the unique partition $\lambda$ and the integer $s\in \mathbb{Z}$
 such that $L_s (\lambda)=L$. Indeed, each part corresponds to a black bead of the abacus with length given by the number of empty positions at its left, the integer $s$ is equal to $x+1$ where $x$ is the position of the rightmost black bead in the abacus obtained after sliding all the black beads as much as possible at the right in $L$.

\begin{exa}
Let us take the partition $\lambda:=(5,4,2,1,1)$ and $s=0$.  The associated abacus $L_{0}(\lambda)$ may be represented as follows, where the positions  at the right 
 of the dashed vertical line are labelled by the non negative  integers:%

\[
\begin{tikzpicture}[scale=0.5, bb/.style={draw,circle,fill,minimum size=2.5mm,inner sep=0pt,outer sep=0pt}, wb/.style={draw,circle,fill=white,minimum size=2.5mm,inner sep=0pt,outer sep=0pt}]
\node [wb] at (11,2) {};
\node [wb] at (10,2) {};
\node [wb] at (9,2) {};
\node [wb] at (8,2) {};
\node [wb] at (7,2) {};
\node [bb] at (6,2) {};
\node [wb] at (5,2) {};
\node [bb] at (4,2) {};
\node [wb] at (3,2) {};
\node [wb] at (2,2) {};
\node [bb] at (1,2) {};
\node [wb] at (0,2) {};
\node [bb] at (-1,2) {};
\node [bb] at (-2,2) {};
\node [wb] at (-3,2) {};
\node [bb] at (-4,2) {};
\node [bb] at (-5,2) {};
\node [bb] at (-6,2) {};
\node [bb] at (-7,2) {};
\node [bb] at (-8,2) {};
\node [bb] at (-9,2) {};
\draw[dashed](1.5,1.5)--node[]{}(1.5,2.5);
\end{tikzpicture}
\]
\end{exa}

To $L_{s}(\lambda)$, one can associate an $e$-tuple of abacus $\mathcal{L}_s^e (\lambda):=(L^0,\ldots,L^{e-1})$. This is done as follows: for each black bead in position $k$ in 
 $L_{s}(\lambda)$, we write $k=q.e+r$  where $q\in \mathbb{Z}$ and $r\in \{0,\ldots,e-1\}$ and we set a black bead in position $q$ of the abacus $L_r$. To picture this, write first the abacus $L_0$,
  then immediately above the abacus $L_1$ and so on, so that all the beads associated with the entry $0$ of each abacus appear in the same vertical line. 

\begin{exa}
For the partition $(5.4.2.1.1)$ and $e=3$, we get the following: 
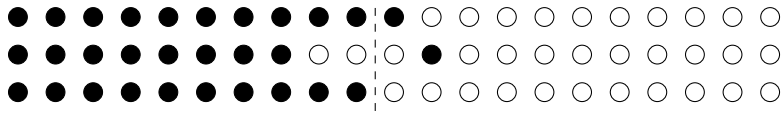
\begin{center}
\begin{tikzpicture}[scale=0.5, bb/.style={draw,circle,fill,minimum size=2.5mm,inner sep=0pt,outer sep=0pt}, wb/.style={draw,circle,fill=white,minimum size=2.5mm,inner sep=0pt,outer sep=0pt}]

	\node [wb] at (11,0) {};
	\node [wb] at (10,0) {};
	\node [wb] at (9,0) {};
	\node [wb] at (8,0) {};
	\node [wb] at (7,0) {};
	\node [wb] at (6,0) {};
	\node [wb] at (5,0) {};
	\node [wb] at (4,0) {};
	\node [wb] at (3,0) {};
	\node [wb] at (2,0) {};
	\node [wb] at (1,0) {};
	\node [bb] at (0,0) {};
	\node [bb] at (-1,0) {};
	\node [bb] at (-2,0) {};
	\node [bb] at (-3,0) {};
	\node [bb] at (-4,0) {};
	\node [bb] at (-5,0) {};
	\node [bb] at (-6,0) {};
	\node [bb] at (-7,0) {};
	\node [bb] at (-8,0) {};
	\node [bb] at (-9,0) {};
	
	\node [wb] at (11,1) {};
	\node [wb] at (10,1) {};
	\node [wb] at (9,1) {};
	\node [wb] at (8,1) {};
	\node [wb] at (7,1) {};
	\node [wb] at (6,1) {};
	\node [wb] at (5,1) {};
	\node [wb] at (4,1) {};
	\node [wb] at (3,1) {};
	\node [bb] at (2,1) {};
	\node [wb] at (1,1) {};
	\node [wb] at (0,1) {};
	\node [wb] at (-1,1) {};
	\node [bb] at (-2,1) {};
	\node [bb] at (-3,1) {};
	\node [bb] at (-4,1) {};
	\node [bb] at (-5,1) {};
	\node [bb] at (-6,1) {};
	\node [bb] at (-7,1) {};
	\node [bb] at (-8,1) {};
	\node [bb] at (-9,1) {};
	
	\node [wb] at (11,2) {};
	\node [wb] at (10,2) {};
	\node [wb] at (9,2) {};
	\node [wb] at (8,2) {};
	\node [wb] at (7,2) {};
	\node [wb] at (6,2) {};
	\node [wb] at (5,2) {};
	\node [wb] at (4,2) {};
	\node [wb] at (3,2) {};
	\node [wb] at (2,2) {};
	\node [bb] at (1,2) {};
	\node [bb] at (0,2) {};
	\node [bb] at (-1,2) {};
	\node [bb] at (-2,2) {};
	\node [bb] at (-3,2) {};
	\node [bb] at (-4,2) {};
	\node [bb] at (-5,2) {};
	\node [bb] at (-6,2) {};
	\node [bb] at (-7,2) {};
	\node [bb] at (-8,2) {};
	\node [bb] at (-9,2) {};

	\draw[dashed](0.5,-0.5)--node[]{}(0.5,2.5);
	\end{tikzpicture}
	\end{center}

\end{exa}
For each runner, sliding one black bead from right to left is equivalent to remove an $e$-rim hook in the associated partition. As a consequence, 
 after performing this procedure as much as possible, we obtain an $e$-abacus which can be transformed (by reversing the previous procedure) into an abacus 
   representing  the $e$-core of $\lambda$. The number of moves of the black beads gives the $e$-weight of $\lambda$. 
\begin{exa}
If we do the above procedure for our example, we obtain:
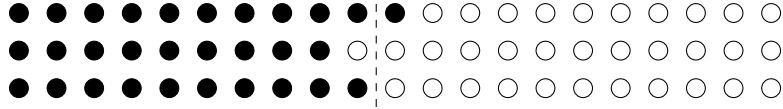
\begin{center}
\begin{tikzpicture}[scale=0.5, bb/.style={draw,circle,fill,minimum size=2.5mm,inner sep=0pt,outer sep=0pt}, wb/.style={draw,circle,fill=white,minimum size=2.5mm,inner sep=0pt,outer sep=0pt}]

	\node [wb] at (11,0) {};
	\node [wb] at (10,0) {};
	\node [wb] at (9,0) {};
	\node [wb] at (8,0) {};
	\node [wb] at (7,0) {};
	\node [wb] at (6,0) {};
	\node [wb] at (5,0) {};
	\node [wb] at (4,0) {};
	\node [wb] at (3,0) {};
	\node [wb] at (2,0) {};
	\node [wb] at (1,0) {};
	\node [bb] at (0,0) {};
	\node [bb] at (-1,0) {};
	\node [bb] at (-2,0) {};
	\node [bb] at (-3,0) {};
	\node [bb] at (-4,0) {};
	\node [bb] at (-5,0) {};
	\node [bb] at (-6,0) {};
	\node [bb] at (-7,0) {};
	\node [bb] at (-8,0) {};
	\node [bb] at (-9,0) {};
	
	\node [wb] at (11,1) {};
	\node [wb] at (10,1) {};
	\node [wb] at (9,1) {};
	\node [wb] at (8,1) {};
	\node [wb] at (7,1) {};
	\node [wb] at (6,1) {};
	\node [wb] at (5,1) {};
	\node [wb] at (4,1) {};
	\node [wb] at (3,1) {};
	\node [wb] at (2,1) {};
	\node [wb] at (1,1) {};
	\node [wb] at (0,1) {};
	\node [bb] at (-1,1) {};
	\node [bb] at (-2,1) {};
	\node [bb] at (-3,1) {};
	\node [bb] at (-4,1) {};
	\node [bb] at (-5,1) {};
	\node [bb] at (-6,1) {};
	\node [bb] at (-7,1) {};
	\node [bb] at (-8,1) {};
	\node [bb] at (-9,1) {};
	
	\node [wb] at (11,2) {};
	\node [wb] at (10,2) {};
	\node [wb] at (9,2) {};
	\node [wb] at (8,2) {};
	\node [wb] at (7,2) {};
	\node [wb] at (6,2) {};
	\node [wb] at (5,2) {};
	\node [wb] at (4,2) {};
	\node [wb] at (3,2) {};
	\node [wb] at (2,2) {};
	\node [bb] at (1,2) {};
	\node [bb] at (0,2) {};
	\node [bb] at (-1,2) {};
	\node [bb] at (-2,2) {};
	\node [bb] at (-3,2) {};
	\node [bb] at (-4,2) {};
	\node [bb] at (-5,2) {};
	\node [bb] at (-6,2) {};
	\node [bb] at (-7,2) {};
	\node [bb] at (-8,2) {};
	\node [bb] at (-9,2) {};

	\draw[dashed](0.5,-0.5)--node[]{}(0.5,2.5);
	\end{tikzpicture}
	\end{center}The associated abacus is 
\[
\begin{tikzpicture}[scale=0.5, bb/.style={draw,circle,fill,minimum size=2.5mm,inner sep=0pt,outer sep=0pt}, wb/.style={draw,circle,fill=white,minimum size=2.5mm,inner sep=0pt,outer sep=0pt}]
\node [wb] at (11,2) {};
\node [wb] at (10,2) {};
\node [wb] at (9,2) {};
\node [wb] at (8,2) {};
\node [wb] at (7,2) {};
\node [wb] at (6,2) {};
\node [wb] at (5,2) {};
\node [bb] at (4,2) {};
\node [wb] at (3,2) {};
\node [wb] at (2,2) {};
\node [bb] at (1,2) {};
\node [wb] at (0,2) {};
\node [bb] at (-1,2) {};
\node [bb] at (-2,2) {};
\node [bb] at (-3,2) {};
\node [bb] at (-4,2) {};
\node [bb] at (-5,2) {};
\node [bb] at (-6,2) {};
\node [bb] at (-7,2) {};
\node [bb] at (-8,2) {};
\node [bb] at (-9,2) {};
\draw[dashed](1.5,1.5)--node[]{}(1.5,2.5);
\end{tikzpicture}
\]
whose associated partition is $(3,1)$  and we have $\omega_3 (\lambda)=3$ as in Example \ref{excore}. 
\end{exa}
Last, we will need an additional notation. For two abaci  $L$ and $L'$, we write $L \subset L'$ if we have the following property: for each  black bead in position $i$ of the abacus 
 $L$, there is a black bead in  position $i$ in $L'$. 

Let now consider ${\bf s}\in \mathbb{Z}^l$ and an $l$-tuple of abaci $(L_{s_1},\ldots,L_{s_l})$. 
 This $l$-abacus is, as above,  conveniently pictured as follows:  take first the 
 abacus $L_{s_1}$ and then just above the abacus $L_{s_2}$  and so on, so that all the beads in position  $0$ of each abacus appear in the same vertical line. 
 \begin{Def}\label{escore1} 
 Under the above notations, we say that the $l$-tuple of abaci $(L_{s_1},\ldots,L_{s_l})$ is {\it $(e,{\bf s})$-complete} if:
\begin{enumerate}
\item $l=1$ and $L_{s_{1}}(\lambda^{1})\subset L_{s_{1}%
+e}(\lambda^{1})$,
\item or $l>1$ and 
 $$L_{s_{1} }(\lambda^1) \subset L_{s_2 }(\lambda^2) \subset \ldots \subset L_{s_l }(\lambda^l)
 \subset L_{s_{1} +e} (\lambda^1).$$
\end{enumerate}
\end{Def}

To a multicharge ${\bf s}\in \mathbb{Z}^l$ and an $l$-partition $\ulambda$, is associated its $l$-abacus defined as the $l$-tuple $(L_{s_1} (\lambda^1),\ldots,L_{s_l} (\lambda^l))$. It can be pictured exactly as above and will be called the {\it $(e,{\bf s})$-abacus} of $\ulambda$. In fact, it does not depend on $e$ but we have chosen here a notation similar to  the notion of $(e,{\bf s})$-core below.

 \begin{exa}\label{ui}
 
 Let ${\bf s}=(0,3)$ and $e=4$. We consider the $2$-partition $((4,1,1),(1,1))$. Its associated $(e,{\bf s})$-abacus 
 $(L _0 (4.1.1),L_1 (1.1))$ can be represented as follows:

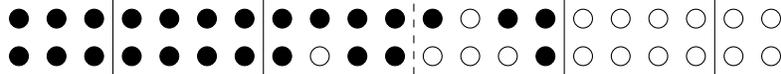
\begin{center}
\begin{tikzpicture}[scale=0.5, bb/.style={draw,circle,fill,minimum size=2.5mm,inner sep=0pt,outer sep=0pt}, wb/.style={draw,circle,fill=white,minimum size=2.5mm,inner sep=0pt,outer sep=0pt}]

	\node [wb] at (11,2) {};
	\node [wb] at (10,2) {};
	\node [wb] at (9,2) {};
	\node [wb] at (8,2) {};
	\node [wb] at (7,2) {};
	\node [wb] at (6,2) {};
	\node [bb] at (5,2) {};
	\node [bb] at (4,2) {};
	\node [wb] at (3,2) {};
	\node [bb] at (2,2) {};
	\node [bb] at (1,2) {};
	\node [bb] at (0,2) {};
	\node [bb] at (-1,2) {};
	\node [bb] at (-2,2) {};
	\node [bb] at (-3,2) {};
	\node [bb] at (-4,2) {};
	\node [bb] at (-5,2) {};
	\node [bb] at (-6,2) {};
	\node [bb] at (-7,2) {};
	\node [bb] at (-8,2) {};
	\node [bb] at (-9,2) {};
	
	\node [wb] at (11,1) {};
	\node [wb] at (10,1) {};
	\node [wb] at (9,1) {};
	\node [wb] at (8,1) {};
	\node [wb] at (7,1) {};
	\node [wb] at (6,1) {};
	\node [bb] at (5,1) {};
	\node [wb] at (4,1) {};
	\node [wb] at (3,1) {};
	\node [wb] at (2,1) {};
	\node [bb] at (1,1) {};
	\node [bb] at (0,1) {};
	\node [wb] at (-1,1) {};
	\node [bb] at (-2,1) {};
	\node [bb] at (-3,1) {};
	\node [bb] at (-4,1) {};
	\node [bb] at (-5,1) {};
	\node [bb] at (-6,1) {};
	\node [bb] at (-7,1) {};
	\node [bb] at (-8,1) {};
	\node [bb] at (-9,1) {};
	
	\draw[](-6.5,0.5)--node[]{}(-6.5,2.5);
	\draw[](-2.5,0.5)--node[]{}(-2.5,2.5);
	\draw[dashed](1.5,0.5)--node[]{}(1.5,2.5);
		\draw[](5.5,0.5)--node[]{}(5.5,2.5);
	\draw[](9.5,0.5)--node[]{}(9.5,2.5);
	\end{tikzpicture}

\end{center}

 \end{exa}

\subsection{ The notion of $(e,{\bf s})$-cores}\label{esc}

The notion of $(e,{\bf s})$-core has been introduced in \cite[Def. 5.7]{JL} (the definition below is slightly different but it is an easy exercice to show the equivalence). This is a generalization of the 
 notion of $e$-core  partitions in the context of  $l$-partitions associated with a multicharge.  First let us introduce the notion of reduced $(e,{\bf s})$-core:

  \begin{Def}\label{escore0} Assume that ${\bf s}\in \mathbb{Z}^l$ then 
we say that the $l$-partition ${\boldsymbol{\lambda}}$ is a  {\it reduced  $(e,\mathbf{s}%
)$-core} if its $(e,{\bf s})$-abacus  $(L_{s_1} (\lambda^1),\ldots, L_{s_l} (\lambda^l))$  is $(e,\mathbf{s}%
)$-complete. 
\end{Def}
To give a first study of this notion, let us introduce  the following two sets:
    $$\overline{\mathcal{A}}^l_e:=\{ (s_1,\ldots,s_l)\in \mathbb{Z}^l\ |\ \forall (i,j)\in \{1,\ldots,l\},\ i<j,\ 0\leq s_j-s_i\leq e\},$$
   $${\mathcal{A}}^l_e:=\{ (s_1,\ldots,s_l)\in \mathbb{Z}^l\ |\ \forall (i,j)\in \{1,\ldots,l\},\ i<j,\ 0\leq s_j-s_i<e\}.$$
   
   \begin{Prop}\label{RedinAbar}
   Assume that ${\bf s}\in \mathbb{Z}^l$ then if ${\boldsymbol{\lambda}}$ is a  reduced $(e,\mathbf{s}%
)$-core, we have ${\bf s}\in\overline{\mathcal{A}}^l_e$. 
   
   \end{Prop}
   \begin{proof}
   Assume that  $(L_{s_1} (\lambda^1),\ldots, L_{s_l} (\lambda^l))$  is $(e,\mathbf{s}%
)$-complete  then for each $i=1,\ldots,l-2$, we have $L_{s_{_i} }(\lambda^i) \subset L_{s_{i+1} }(\lambda^{i+1})$ which implies that $s_{i+1}\geq s_i$. We also have $L_{s_l}\subset L_{s_1+e}$ and this implies that $s_l\leq s_1+e$. This concludes the proof.

   \end{proof}

  We now   give the definition of our main object of interest.
 Let  ${\bf s}\in \mathbb{Z}^l$ and $e\in \mathbb{N}_{>0}$, denote by $\widetilde{\bf s}:=(s_1',\ldots,s_l')\in \{0,\ldots,e-1\}^l$  the multicharge such that $s_i'\equiv s_i (\text{mod }e)$. Then  we define $\sigma_{\bf s}  \in \mathfrak{S}_l$ to be the unique permutation such that
 $$s_{\sigma_{\bf s} (1)} ' \leq s_{\sigma_{\bf s} (2)} ' \leq \ldots \leq s_{\sigma_{\bf s} (l)} ' $$
 with the additional property that if $s_{\sigma_{\bf s} (i)}=s_{\sigma_{\bf s} (i+1)}$ for $i\in \{1,\ldots,l-1\}$ then ${\sigma_{\bf s} (i)}<\sigma_{\bf s} (i+1)$. We set:
 \begin{equation}\label{eq}
 \widetilde{\bf s}^{\sigma_{\bf s}}:=(s_{\sigma_{\bf s} (1)}' ,  s_{\sigma_{\bf s} (2)} ',  \ldots , s_{\sigma_{\bf s} (l)}').
 \end{equation}
  We then clearly have  $\widetilde{\bf s}^{\sigma_{\bf s}}\in \mathcal{A}^l_e$.

\begin{Def}\label{escore2} Let ${\bf s}\in \mathbb{Z}^l$, 
we say that the $l$-partition ${\boldsymbol{\lambda}}$ is a {\it $(e,\mathbf{s}%
)$-core} if    the $l$-partition $\ulambda^{\sigma_{\bf s}}:= ( \lambda^{\sigma_{\bf s} (1)},\ldots,\lambda^{\sigma_{\bf s} (l)})$ 
 is a reduced $(e,\widetilde{\bf s}^{{\sigma_{\bf s}}})$-core. 
We denote by $\mathfrak{C}^l (e,\mathbf{s})$ the set of all $(e,\mathbf{s})$-cores. 
\end{Def}

As already noted in the previous paragraph, for $l=1$, the $(e, {s}%
)$-core are exactly the   $e$-cores. Thus, the set $\mathfrak{C}^1 (e,{s})$ does not depend on 
 $s\in \mathbb{Z}$ and is exactly  given by the set of $e$-cores $\mathfrak{C} (e)$. One can also  easily see that if $\ulambda$ is a $(e,\mathbf{s}%
)$-core, each component $\lambda^j$ is an $e$-core. 

\begin{Rem}
Assume that ${\bf s}\in\overline{\mathcal{A}}^l_e$ and there exists $i\in \{1,\ldots,l-1\}$ such that $s_i=s_{i+1}$. Then if 
$\ulambda$ is a reduced $(e,\mathbf{s}%
)$-core we must have $\lambda^i=\lambda^{i+1}$. 
\end{Rem}

We need to check  that  the  reduced $(e,\mathbf{s}%
)$-cores are always  $(e,\mathbf{s})$-cores. This is clear if ${\bf s}\in{\mathcal{A}}^l_e$ but not if 
 ${\bf s}\in\overline{\mathcal{A}}^l_e \setminus {\mathcal{A}}^l_e$. So 
let us assume that ${\bf s}\in\overline{\mathcal{A}}^l_e$ but ${\bf s}\notin{\mathcal{A}}^l_e$ and let $\ulambda$ be a reduced $(e,\mathbf{s}
)$-core, this implies that there exists $j\in \{2,\ldots,l\}$ such that $s_j=s_{j+1}=\ldots=s_l=s_1+e$. Then 
the abacus $(L_{s_j-e} (\lambda^l),\ldots,L_{s_l-e} (\lambda^l), L_{s_1} (\lambda^1),\ldots, L_{s_{j-1}} (\lambda^{l-1}))$ is $(e,(s_j-e,\ldots,s_l-e,s_1,\ldots,s_{j-1}))$-complete.
 By the above remark, we thus obtain $\lambda^{j}=\ldots=\lambda^{l}=\lambda^1$.  We so conclude that 
   in the case where ${\bf s}\in\overline{\mathcal{A}}^l_e$, the $(e,\mathbf{s}
)$-cores are exactly the reduced $(e,\mathbf{s}
)$-cores.

\begin{Rem}
The above definition can be formulated in terms of $\beta$-numbers and symbols (see \cite[\S 5.1]{JL}), which gives an equivalent definition of the set of $(e,{\bf s})$-cores.  We get that 
 $\ulambda$ is a $(e,{\bf s})$-core if and only if
 \begin{itemize}
\item  for all $c=1,\ldots,l-1$ and $j\in \mathbb{Z}_{>0}$, there exists $i\in \mathbb{Z}_{>0}$ such that 
$$\lambda^{\sigma_{\bf s} (c)}_j-j+s_{\sigma_{\bf s} (c)}' =\lambda^{\sigma_{\bf s} (c+1)}_i-i+s_{\sigma_{\bf s} (c+1)}', $$

\item for all  $j\in \mathbb{Z}_{>0}$, there exists $i\in \mathbb{Z}_{>0}$ such that 
$$\lambda^{\sigma_{\bf s} (l)}_j-j+s_{\sigma_{\bf s} (l)}' =\lambda^{\sigma_{\bf s} (1)}_i-i+s_{\sigma_{\bf s} (1)}' +e.$$
\end{itemize}

\end{Rem}

\begin{Rem}\label{RT0}
As already noticed, the irreducible representations of the Ariki-Koike algebras associated with the datum $(e,{\bf s})$ are 
 naturally labeled by a distinguished set of $l$-partitions called Uglov $l$-partitions.  In the particular case where  
 ${\bf s}\in \mathcal{A}^l_e$, these $l$-partitions are called FLOTW $l$-partitions and it is easy to check that any $(e,\mathbf{s}%
)$-core is then a FLOTW  $l$-partition in the sense of \cite[Th. 5.8.5]{GJ}. Now for an arbitrary choice of ${\bf s}$, there is an
 explicit bijection between  the set of FLOTW partitions  associated with $(e,\widetilde{\bf s}^{{\sigma_{\bf s}}})$  
 and   the set of Uglov $l$-partitions associated with  $(e,{\bf s})$
 (this bijection is described in \cite{JL2}).  
 It is easy to see that this bijection restricted to the set of 
 $(e,{\bf s})$-cores sends $\ulambda$ to $\ulambda^{\sigma_{\bf s}^{-1}}$. 
   This implies that  $(e,{\bf s})$-cores  are always Uglov $l$-partitions. This fact has a representation theoretic meaning as we will see in the following. 

\end{Rem}

\begin{exa}
Let $l=2$, $e=3$ and ${\bf s}=(0,1)$. Consider the $2$-partition $((1,1),(3,1,1))$. With the above notation, we have $\sigma=\text{Id}$ and ${\bf s}'={\bf s}={\bf s}^{\sigma_{\bf s}}$.  The associated $2$-abacus is

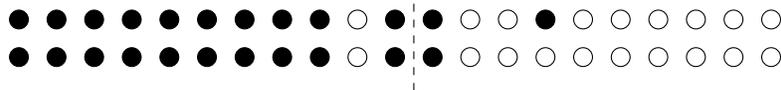
\begin{center}
\begin{tikzpicture}[scale=0.5, bb/.style={draw,circle,fill,minimum size=2.5mm,inner sep=0pt,outer sep=0pt}, wb/.style={draw,circle,fill=white,minimum size=2.5mm,inner sep=0pt,outer sep=0pt}]

	\node [wb] at (11,2) {};
	\node [wb] at (10,2) {};
	\node [wb] at (9,2) {};
	\node [wb] at (8,2) {};
	\node [wb] at (7,2) {};
	\node [wb] at (6,2) {};
	\node [bb] at (5,2) {};
	\node [wb] at (4,2) {};
	\node [wb] at (3,2) {};
	\node [bb] at (2,2) {};
	\node [bb] at (1,2) {};
	\node [wb] at (0,2) {};
	\node [bb] at (-1,2) {};
	\node [bb] at (-2,2) {};
	\node [bb] at (-3,2) {};
	\node [bb] at (-4,2) {};
	\node [bb] at (-5,2) {};
	\node [bb] at (-6,2) {};
	\node [bb] at (-7,2) {};
	\node [bb] at (-8,2) {};
	\node [bb] at (-9,2) {};
	
	\node [wb] at (11,1) {};
	\node [wb] at (10,1) {};
	\node [wb] at (9,1) {};
	\node [wb] at (8,1) {};
	\node [wb] at (7,1) {};
	\node [wb] at (6,1) {};
	\node [wb] at (5,1) {};
	\node [wb] at (4,1) {};
	\node [wb] at (3,1) {};
	\node [bb] at (2,1) {};
	\node [bb] at (1,1) {};
	\node [wb] at (0,1) {};
	\node [bb] at (-1,1) {};
	\node [bb] at (-2,1) {};
	\node [bb] at (-3,1) {};
	\node [bb] at (-4,1) {};
	\node [bb] at (-5,1) {};
	\node [bb] at (-6,1) {};
	\node [bb] at (-7,1) {};
	\node [bb] at (-8,1) {};
	\node [bb] at (-9,1) {};

	\draw[dashed](1.5,0.1)--node[]{}(1.5,2.5);
	\end{tikzpicture}
	\end{center}
and we see that we here have a $(e,\mathbf{s})$-core.  As a consequence, taking ${\bf s}=(10,0)$, we have that the 
 $2$-partition $((3,1,1),(1,1))$ is a $(e,{\bf s})$-core.

\end{exa}

\subsection{Uglov map}
Let ${\bf s}\in \overline{\mathcal{A}}^l_e$. 
We now show how to associate to a reduced $(e,\mathbf{s})$-core $\ulambda$ a certain $e$-core partition that we denote by   $\tau_{e,{\bf s}} (\ulambda)$ and conversely.  
This construction uses a map defined by  Uglov \cite[\S 4.1]{Ug} (see also \cite[\S 3.1]{Yv} ) which associates a partition to any charged $l$-partition. We will be interested in the restriction of this map to the set of reduced 
$(e,\mathbf{s})$-cores. 

Let $\ulambda$ be an $l$-partition.    
We consider the  $l$-abacus $(L_{s_{1}} (\lambda^1),\ldots,L_{s_l} (\lambda^l))$. Then we construct an associated $1$-abacus as follows. For each $c=1,\ldots,l$ and for each  black bead in position $k$ of the abacus $L_{s_{c} }$, we write
$$k=q.e+r$$
with $q\in \mathbb{Z}$ and $r\in \{0,\ldots,e-1\}$. Then we set a black bead in our new abacus in position $(l-c)e+qel+r$.  We  then define   $\tau_{e,{\bf s}} (\ulambda)$ to be the partition associated with this resulting abacus.  We obtain a map 
$$\tau_{e,{\bf s}}:\Pi^l \to \Pi^1$$
which will be called the {\it Uglov map}. Let us illustrate the computation of the Uglov map by two following examples. 
 
\begin{exa}\label{ex3}

We resume Example \ref{ui}. The above procedure gives the following abacus:

\[
\begin{tikzpicture}[scale=0.5, bb/.style={draw,circle,fill,minimum size=2.5mm,inner sep=0pt,outer sep=0pt}, wb/.style={draw,circle,fill=white,minimum size=2.5mm,inner sep=0pt,outer sep=0pt}]
\node [wb] at (11,2) {};
\node [wb] at (10,2) {};
\node [wb] at (9,2) {};
\node [wb] at (8,2) {};
\node [wb] at (7,2) {};
\node [wb] at (6,2) {};
\node [bb] at (5,2) {};
\node [wb] at (4,2) {};
\node [wb] at (3,2) {};
\node [wb] at (2,2) {};
\node [bb] at (1,2) {};
\node [bb] at (0,2) {};
\node [wb] at (-1,2) {};
\node [bb] at (-2,2) {};
\node [bb] at (-3,2) {};
\node [bb] at (-4,2) {};
\node [wb] at (-5,2) {};
\node [bb] at (-6,2) {};
\node [bb] at (-7,2) {};
\node [bb] at (-8,2) {};
\node [bb] at (-9,2) {};
	\draw[dashed](-2.5,1.5)--node[]{}(-2.5,2.5);
\end{tikzpicture}
\]
We thus get $\tau_{e,{\bf s}} (\ulambda)=(5,2,2,1,1,1)$.

\end{exa}

\begin{exa}
Let ${\bf s}=(0,1,2)$ and $e=4$. We consider the $3$-partition $((2),(1),(1,1))$, the associated $3$-abacus 
 $(L _0 (2),L_1 (1), L_2 (1,1))$ can be written as:

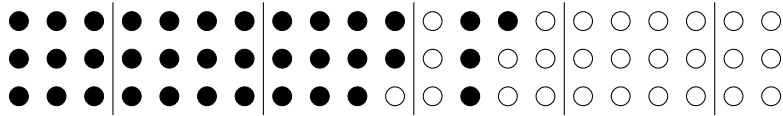
\begin{center}
\begin{tikzpicture}[scale=0.5, bb/.style={draw,circle,fill,minimum size=2.5mm,inner sep=0pt,outer sep=0pt}, wb/.style={draw,circle,fill=white,minimum size=2.5mm,inner sep=0pt,outer sep=0pt}]

	\node [wb] at (11,2) {};
	\node [wb] at (10,2) {};
	\node [wb] at (9,2) {};
	\node [wb] at (8,2) {};
	\node [wb] at (7,2) {};
	\node [wb] at (6,2) {};
	\node [wb] at (5,2) {};
	\node [bb] at (4,2) {};
	\node [bb] at (3,2) {};
	\node [wb] at (2,2) {};
	\node [bb] at (1,2) {};
	\node [bb] at (0,2) {};
	\node [bb] at (-1,2) {};
	\node [bb] at (-2,2) {};
	\node [bb] at (-3,2) {};
	\node [bb] at (-4,2) {};
	\node [bb] at (-5,2) {};
	\node [bb] at (-6,2) {};
	\node [bb] at (-7,2) {};
	\node [bb] at (-8,2) {};
	\node [bb] at (-9,2) {};
	
	\node [wb] at (11,1) {};
	\node [wb] at (10,1) {};
	\node [wb] at (9,1) {};
	\node [wb] at (8,1) {};
	\node [wb] at (7,1) {};
	\node [wb] at (6,1) {};
	\node [wb] at (5,1) {};
	\node [wb] at (4,1) {};
	\node [bb] at (3,1) {};
	\node [wb] at (2,1) {};
	\node [bb] at (1,1) {};
	\node [bb] at (0,1) {};
	\node [bb] at (-1,1) {};
	\node [bb] at (-2,1) {};
	\node [bb] at (-3,1) {};
	\node [bb] at (-4,1) {};
	\node [bb] at (-5,1) {};
	\node [bb] at (-6,1) {};
	\node [bb] at (-7,1) {};
	\node [bb] at (-8,1) {};
	\node [bb] at (-9,1) {};
	
	\node [wb] at (11,0) {};
	\node [wb] at (10,0) {};
	\node [wb] at (9,0) {};
	\node [wb] at (8,0) {};
	\node [wb] at (7,0) {};
	\node [wb] at (6,0) {};
	\node [wb] at (5,0) {};
	\node [wb] at (4,0) {};
	\node [bb] at (3,0) {};
	\node [wb] at (2,0) {};
	\node [wb] at (1,0) {};
	\node [bb] at (0,0) {};
	\node [bb] at (-1,0) {};
	\node [bb] at (-2,0) {};
	\node [bb] at (-3,0) {};
	\node [bb] at (-4,0) {};
	\node [bb] at (-5,0) {};
	\node [bb] at (-6,0) {};
	\node [bb] at (-7,0) {};
	\node [bb] at (-8,0) {};
	\node [bb] at (-9,0) {};
	
	\draw[](-6.5,-0.5)--node[]{}(-6.5,2.5);
	\draw[](-2.5,-0.5)--node[]{}(-2.5,2.5);
	\draw[](1.5,-0.5)--node[]{}(1.5,2.5);
		\draw[](5.5,-0.5)--node[]{}(5.5,2.5);
	\draw[](9.5,-0.5)--node[]{}(9.5,2.5);
	\end{tikzpicture}

\end{center}
The above procedure gives the following abacus:

\[
\begin{tikzpicture}[scale=0.5, bb/.style={draw,circle,fill,minimum size=2.5mm,inner sep=0pt,outer sep=0pt}, wb/.style={draw,circle,fill=white,minimum size=2.5mm,inner sep=0pt,outer sep=0pt}]
\node [wb] at (11,2) {};
\node [wb] at (10,2) {};
\node [wb] at (9,2) {};
\node [bb] at (8,2) {};
\node [wb] at (7,2) {};
\node [wb] at (6,2) {};
\node [wb] at (5,2) {};
\node [bb] at (4,2) {};
\node [wb] at (3,2) {};
\node [wb] at (2,2) {};
\node [bb] at (1,2) {};
\node [bb] at (0,2) {};
\node [wb] at (-1,2) {};
\node [wb] at (-2,2) {};
\node [bb] at (-3,2) {};
\node [bb] at (-4,2) {};
\node [bb] at (-5,2) {};
\node [bb] at (-6,2) {};
\node [bb] at (-7,2) {};
\node [bb] at (-8,2) {};
\node [bb] at (-9,2) {};
	\draw[dashed](-1.5,1.5)--node[]{}(-1.5,2.5);
\end{tikzpicture}
\]
We thus get $\tau_{e,{\bf s}} (\ulambda)=(7,4,2,2)$.

\end{exa}

%
%
%
%

The  map  $\tau_{e,{\bf s}}$ is not surjective in general but it is clearly injective. 
%
%

 \begin{Prop}\label{taucore}
 Let ${\bf s}\in\overline{\mathcal{A}}^l_e$, then $\tau_{e,{\bf s}} (\uemptyset)$ is an $e$-core.
 \end{Prop}
 \begin{proof}
 This is clear by the characterization of $e$-cores with abaci in the last section. 
 
 \end{proof}

 \begin{Prop}\label{emain}

 The map 
  $$\begin{array}{cccc}
 \tau^l_{e}: & \{ (\ulambda,{\bf s})\ |\ {\bf s}\in \overline{\mathcal{A}}^l_e,\ \ulambda \in  \mathfrak{C}^l (e,\mathbf{s}) \} & \to & 
   \{ (\lambda,s)\ |\ s\in \mathbb{Z} ,\ \lambda \in \mathfrak{C}^1 (e) \}\\
                   & (\ulambda,{\bf s}) & \mapsto & (\tau_{e,{\bf s}} (\ulambda) ,\sum_{1\leq i\leq l} s_i)
                   \end{array}
 $$
 is bijective.
 \end{Prop}
 \begin{proof}
 First, the map is well defined. Indeed, assume  that $\ulambda \in \mathfrak{C}^l (e,\mathbf{s})$ with ${\bf s}\in \overline{\mathcal{A}}^l_e$. Then 
  $\ulambda$ satisfies  the property in Definition \ref{escore1} $(2)$  but this implies that 
    the partition $\tau_{e,{\bf s}} (\ulambda) $ satisfies $(1)$ of Definition \ref{escore1}. We deduce that it is an $e$-core as desired.  Now let us prove that the map is bijective.  Let $s\in \mathbb{Z}$  and $\lambda \in \mathfrak{C}^1 (e)$. Then we have an associated $(e,s)$-abacus associated with this datum and by construction, there exists a unique $\ulambda\in \Pi^l$ and 
     $ {\bf s}\in \overline{\mathcal{A}}^l_e$ such that $\tau_{e,{\bf s}} (\ulambda)=\lambda$ and 
 $\sum_{1\leq i\leq l} s_i=s$. It thus suffices to prove that $\ulambda\in \mathfrak{C}^l (e,\mathbf{s})$. But it follows from the fact that 
  its $(e,{\bf s})$-abacus is complete because $\lambda$ is a $e$-core.

 \end{proof}
 
 \begin{Rem}
 If we consider ${\bf s}\notin \overline{\mathcal{A}}^l_e$  and a $(e,{\bf s})$-core $\ulambda$ 
  then we have  $\tau_{e,{\bf s}} (\ulambda)\notin  \mathfrak{C}^1 (e)$ in general. 
 
 \end{Rem}

We now give two important results showing remarkable links between $\ulambda$ and 
 $\tau_{e,{\bf s}} (\ulambda)$. The first one compare the number of nodes in the two Young diagrams with a given residue.
 
  \begin{Prop}\label{compare}
   Let $\ulambda\in \Pi^l$ and ${\bf s}\in\overline{\mathcal{A}}^l_e$. Set $s=\sum_{1\leq i\leq l} s_i$.
For all $i=0,1,\ldots,e-1$, we have:
$$c^{(e, s) }_i (\tau_{e,{\bf s}} (\ulambda) )-c^{(e,s)}_i (\tau_{e,{\bf s}} (\uemptyset))=c^{e,{\bf s}}_i (\ulambda)+l. c^{e,{\bf s}}_0 (\ulambda)$$
 \end{Prop}
 \begin{proof}
 We will argue by induction on the rank of $\ulambda$. If this rank is $0$ then $\ulambda$ is the empty $l$-partition and the result is trivial. Assume now that $\ulambda$ is an $l$-partition of rank $n>0$.
  Let $\umu$ be an $l$-partition of rank $n-1$ which is obtained from $\ulambda$ by removing a removable $i$-node for some $i\in \mathbb{Z}/e\mathbb{Z}$.
   Assume first that $i\neq 0 (\textrm{mod }e)$. Then it is easy to see that   $\tau_{e,{\bf s}} (\umu)$ is obtained from   $\tau_{e,{\bf s}} (\ulambda)$  by removing a removable $i$-node. As a consequence,
    we have  $c^{(e,s) }_j (\tau_{e,{\bf s}} (\ulambda) )=c^{(e,s) }_j (\tau_{e,{\bf s}} (\umu) )$ and 
     $c^{e,{\bf s}}_j (\ulambda)=c^{e,{\bf s}}_j (\umu)$ if $j$ is different from $i$ modulo $e\mathbb{Z}$. Thus, we get $c^{(e,s) }_i (\tau_{e,{\bf s}} (\ulambda) )=c^{(e,s) }_i (\tau_{e,{\bf s}} (\umu) )+1$ and 
     $c^{e,{\bf s}}_i (\ulambda)=c^{e,{\bf s}}_i (\umu)+1$. So the formula is still true by induction. 
     
     Assume now that $i=0(\textrm{mod }e)$. In this case, we still  have $c^{e,{\bf s}}_j (\ulambda)=c^{e,{\bf s}}_j (\umu)$ if $j\neq 0$ and  $c^{e,{\bf s}}_0 (\ulambda)=c^{e,{\bf s}}_0 (\umu)+1$. 
    Now, we need to see how $\tau_{e,{\bf s}} (\ulambda)$ is obtained from $\tau_{e,{\bf s}} (\umu)$.  The node that we add to $\umu$ to obtain $\ulambda$ corresponds to a black bead in the abacus 
     of $\tau_{e,{\bf s}} (\ulambda)$ and to another in the abacus of  $\tau_{e,{\bf s}} (\umu)$. 
     Let us denote by $m$ the number of black beads between  theses two positions (not including these two)  in the abacus (the number is the same in both abaci).
     Then $\tau_{e,{\bf s}} (\ulambda)$   is obtained by removing a part of length $x>0$ ending by a node with residue 
     $e-1$, adding one node  to the $m$ parts above and adding one part of length $x+l.e-m+1$ which ends with a node with residue $0$. This thus consists in $x+l.e+1$ consecutive nodes.
     More precisely, to obtain $\tau_{e,{\bf s}} (\ulambda)$  from $\tau_{e,{\bf s}} (\umu)$ 
      we add 
     $l+1$ nodes with residue $0$, and $l$ nodes of residue $j$ for all $j\neq 0$. 
      Thus we obtain 
      $$c^{(e,s) }_i (\tau_{e,{\bf s}} (\ulambda) )=c^{(e,s) }_i (\tau_{e,{\bf s}} (\umu) )+l$$
       if $i\neq 0$ and 
            $$c^{(e,s) }_0 (\tau_{e,{\bf s}} (\ulambda) )=c^{(e,s) }_0 (\tau_{e,{\bf s}} (\umu) )+l+1.$$ 
     Now we have by induction for all $i\in \{0,\ldots,e-1\}$:
     $$c^{(e,s) }_i (\tau_{e,{\bf s}} (\umu) )-c^{(e,s)}_i (\tau_{e,{\bf s}} (\uemptyset))=c^{e,{\bf s}}_i (\umu)+l. c^{e,{\bf s}}_0 (\umu)$$
 which permits to conclude. 
 \end{proof}

 Recall the notation  $\mathcal{C}_{e,{\bf s}} (\ulambda) $ introduced in Subsection \ref{subset_multi} for the multiset of residues of a multipartition.
 \begin{Cor}\label{block}   Let $\ulambda\in \Pi^l$, $\umu \in \Pi^l$  and ${\bf s}\in \overline{\mathcal{A}}^l_e$. 
 We have 
 $$\mathcal{C}_{e,s} (\tau_{e,{\bf s}} (\ulambda)) =\mathcal{C}_{e,s}  (\tau_{e,{\bf s}} (\umu))  \iff \mathcal{C}_{e,{\bf s}} (\ulambda) =\mathcal{C}_{e,{\bf s}} (\umu).$$
 \end{Cor}
\begin{proof}
This directly follows from the previous proposition. 
\end{proof}

Last, we will need a useful property  which permits to compare the number of removable and addable $i$-nodes of $\lambda$ and 
 $\tau_{e,{\bf s}} (\ulambda)$. To do this, we denote  by $M_i^{\bf s} (\ulambda) $
 the number of addable nodes of $\ulambda$ minus the number of removable nodes of $\ulambda$. 
\begin{Prop}\label{compa}
For all $\ulambda\in \Pi^l$,   ${\bf s}\in\overline{\mathcal{A}}^l_e$ and $i\in \mathbb{Z}/e\mathbb{Z}$, we have:
$$M_i^{\bf s} (\ulambda)=\left\{
\begin{array}{rl}
M_i^{s} (\tau_{e,{\bf s}} (\ulambda)) & \text{ if }i\neq 0, \\
M_i^{s} (\tau_{e,{\bf s}} (\ulambda))+l-1 & \text{ if }i=0. 
\end{array}\right.$$
\end{Prop}
\begin{proof}
First, consider a partition $\lambda$ and a charge $s$ and write its associated $1$-abacus. Let $i\in \mathbb{Z}/e\mathbb{Z}$.  
Let $x \in \mathbb{Z}$ be such that $x\equiv i (\text{mod }e)$.  Note that each black bead in 
 the abacus corresponds to a part $\lambda_i$ of the partition $\lambda$ (the position of this bead being given by $\lambda_i-i+s$).
\begin{itemize}
\item If we have a black bead in position $x$ and a black bead in position $x-1$, this does not correspond to any removable nor addable $i$-node.
\item If we have a black bead in position $x$ and no black bead in position $x-1$, this does  correspond to one removable  $i$-node.
\item If we have no black bead in position $x$ and a black bead in position $x-1$, this does  correspond to one addable $i$-node.
\end{itemize}
Of course no black bead  in position $x$ and no black bead in position $x-1$ means that we have no associated addable or removable $i$-node. 

Let us fix $r<<0$ and  let us now consider all the black beads in position greater  (or equal) than $r.e$ in the abacus, for each $i\in \mathbb{Z}/e\mathbb{Z}$, write 
    $B^r_i (\lambda,s)$ the number of such black beads in position $x$ in the abacus with $x\equiv i (\text{mod }e)$. This number is finite by assumption. 
The above discussion shows that:
$$M_i^{s} (\lambda)=\left\{
\begin{array}{rl}
B^r_{i-1}(\lambda,s) -B^r_{i}(\lambda,s)  & \text{ if }i\neq 0 \\
B^r_{i-1}(\lambda,s)  -B^r_{i}(\lambda,s)  +1  & \text{ if }i=0 
\end{array}\right.$$
(the last equality comes from the fact that we have a black bead in position $r.e-1$). 

Now let $(\ulambda,{\bf s})\in \Pi^l \times \mathbb{Z}^l$. We fix again $r<<0$, by the discussion above, for each $c\in \{1,\ldots,l\}$ and 
 $i\in \mathbb{Z}/e\mathbb{Z}$, we have:  
  $$M_i^{_{s_{c} }} (\lambda^{s_{c}})=\left\{
\begin{array}{rl}
B^r_{i-1} ( \lambda^{s_{c}}, s_{c}) -B^r_{i}( \lambda^{s_{c}}, s_{c}) & \text{ if }i\neq 0 \\
B^r_{i-1}( \lambda^{s_{c}}, s_{ c}) -B^r_{i}( \lambda^{s_{c}}, s_{c})  +1  & \text{ if }i=0 
\end{array}\right.$$
By construction, we obtain for all $i\in \mathbb{Z}/e\mathbb{Z}$

$$ B^{lr}_{i}(\tau_{e,{\bf s}} (\lambda),s) =\sum_{c=1,\ldots,l} B^r_{i} ( \lambda^{s_{c}}).$$
As in addition, we also have:
$$M_i^{\bf s} (\ulambda)=\sum_{c=1,\ldots,l} M_i^{s_{c} } ( \lambda^{s_{c}}),$$
we can conclude. 
\end{proof}
\begin{exa}
Let us illustrate the proof with the $2$-partition $((4,1,1),(1,1))$ and the multicharge $(0,3)$  of Example \ref{ex3} (here $e=4$). The Young diagram with its residues is:
  $$
\left(
\begin{array}{|c|c|c|c|}
  \hline
  0& 1  &2 &3   \\
  \hline
  3 \\
  \cline{1-1}
  2 \\
  \cline{1-1}
\end{array}\;,\;
\begin{array}{|c|c|}
  \hline
  3        \\
  \cline{1-1}
 2   \\
 \cline{1-1}
\end{array}
\right)$$
We have seen that  $\tau_{e,{\bf s}} ((4,1,1),(1,1))=(5,2,2,1,1,1)$ with $s=0+3=3$. Thus the associated Young diagram with  residues is:
  $$
\begin{array}{|c|c|c|c|c|}
  \hline
  3&0  &1 & 2 &  3  \\
  \cline{1-5}
  2 &3 \\
  \cline{1-2}
   1 & 2   \\
      \cline{1-2}
        0 \\
  \cline{1-1}
        3  \\
  \cline{1-1}
      2  \\
    \cline{1-1}    
\end{array}$$
On the one hand, we have $M^{(0,3)}_0 ((4,1,1),(1,1))=3$ and $M^3_0 (5,2,2,1,1,1)=2$. On the other hand, we get $M^{(0,3)}_1 ((4,1,1),(1,1))=M^3_1 (5,2,2,1,1,1)=2$,
 $M^{(0,3)}_2 ((4,1,1),(1,1))=M^3_2 (5,2,2,1,1,1)=-2$ and $M^{(0,3)}_3 ((4,1,1),(1,1))=M^3_3 (5,2,2,1,1,1)=-1$.

\end{exa}

\section{Cores and  weights for Ariki-Koike algebras}
In this section, we review the notion of weight for Ariki-Koike algebras as introduced by Fayers in \cite{Fa}. To avoid a possible confusion with the notion of weight for the type $A$ affine Kac-Moody algebra, Fayers weights will be refereed as core-weights in the sequel.  We will notably interpret them in the representation theory of the type $A$ affine Kac-Moody algebra.

\subsection{Block weights for Ariki-Koike algebras and relations with  Fock spaces}

The block weight of an $l$-partition for a given multicharge is defined in \cite{Fa} as follows.
\begin{Def}
Let ${\bf s} \in \mathbb{Z}^l$, $e\in \mathbb{Z}_{>0}$ and $\ulambda \in \Pi^l$, then the {\it block $(e,{\bf s})$-weight} (or simply block weight) of 
 $\ulambda$ is 
 $$p_{(e,{\bf s})} (\ulambda)=\sum_{1\leq i\leq l} c^{e,{\bf s}}_{s_i} (\ulambda) -\frac{1}{2} \sum_{i\in \mathbb{Z}/e\mathbb{Z}} (c^{e,{\bf s}}_i (\ulambda)-c^{e,{\bf s}}_{i-1} (\ulambda))^2.$$
\end{Def}
\begin{Rem}\label{suffw}
From this definition, it is immediate to see that, under the notation of (\ref{eq}), we have for all $\ulambda \in \Pi^l$,
$$p_{(e,{\bf s})} (\ulambda)=p_{(e,\widetilde{\bf s}^{\sigma_{\bf s}})} (\ulambda^{\sigma_{\bf s}})$$
We can thus again restrict ourselves to the case  ${\bf s} \in \mathcal{A}_e^l$.
\end{Rem}

This notion of block weight has a natural interpretation in the representation theory of Kac-Moody algebras that we shall now make explicit. Consider the Kac-Moody algebra  $\mathfrak{g}$  of type $A^{(1)}_{e-1}$. Let $\mathfrak{h}$ be a $\mathbb{Q}$-vector space with basis $\{h_0,\ldots,h_{e-1},D\}$. Let $\{\Lambda_0,\ldots,\Lambda_{e-1},\delta\}$ be the dual basis with respect to the pairing:
  $$\langle .,. \rangle: \mathfrak{h}^* \times \mathfrak{h} \to \mathbb{Q} $$
  defined by:
  $$\langle \Lambda_i,h_j\rangle=\delta_{ij},\ \langle \Lambda_i,D \rangle=\langle \delta, h_i \rangle=0,\ \langle \delta,D\rangle=1\ (0\leq i,j\leq e-1).$$
  The $\Lambda_i$ with $0\leq i\leq e-1\}$ are called the {\it fundamental weights}. The {\it simple roots} $\alpha_i$ with $1\leq i \leq e-1$ are the elements of $\mathfrak{h}^*$ defined by:
  $$\alpha_i:=-\Lambda_{i-1}+2\Lambda_i -\Lambda_{i+1} +\delta_{i,0}\delta$$
  where the subscript have to be understood modulo $e$. For $0\leq i,j \leq e-1$, we denote by  $a_{ij}$ the coefficient of $\Lambda_j$ in $\alpha_i$. Then 
   the matrix $A:=(a_{ij})_{0\leq i,j\leq e-1}$ is the {\it Cartan matrix} of $\widehat{\mathfrak{sl}}_e$. 
    As $(\Lambda_0,\alpha_0,\ldots,\alpha_{e-1})$ is a basis of $\mathfrak{h}^*$, one can define a symmetric non degenerate bilinear form  on 
 $\mathfrak{h}^*$ by setting:
 $$(\alpha_i,\alpha_j)=a_{ij},\ (\Lambda_0,\alpha_i)=\delta_{i,0}, (\Lambda_0,\Lambda_0)=0\ (0\leq i,j\leq e-1\}.$$
 We then derive 
$$(\Lambda_i,\alpha_j)=\delta_{i,j},\ (\delta,\alpha_i)=0\ (0\leq i,j\leq e-1),$$
where $\delta=\alpha_0+\ldots +\alpha_{e-1}$ is the null root.  We have $(\delta,\delta)=0$ and $(\delta,\Lambda_i)=1$ for all $0\leq i\leq e-1$.

Let now consider $v$ an indeterminate and write  $\mathcal{U}_v  (\widehat{\mathfrak{sl}}_e)$ for the {\it quantum affine algebra} of type $A^{(1)}_{e-1}$.  
This is an algebra over $\mathbb{Q}(q)$  with generators $e_i$, $f_i$, $t_i^{\pm 1}$ ($0\leq i\leq e-1$) and $\delta$,  the relations will be omitted (see \cite[Def. 6.1.3]{GJ}).
Fix 
${\bf s}\in \mathbb{Z}^l$ and consider the associated Fock space
$$\mathcal{F}_{\bf s}:=\bigoplus_{\ulambda \in \Pi^l} \mathbb{Q} (v) \ulambda$$ with basis the $l$-partitions.
There is a simple $\mathcal{U}_v  (\widehat{\mathfrak{sl}}_e)$-action on 
$\mathcal{F}_{\bf s}$ (depending on ${\bf s}$) which endows it with the structure of an integrable $\mathcal{U}_v  (\widehat{\mathfrak{sl}}_e)$-module (see \cite[Th. 2.1]{Ug}). In particular, this means that 
 $\mathcal{F}_{\bf s}$  is the direct sum of its weight subspaces. The elements of the basis $\ulambda \in\Pi^l$ are weight vectors whose weights can easily be calculated as follows:
 $$\alpha^{e,{\bf s}} (\ulambda):=-\Delta_{\bf s} \delta +\Lambda_{s_1}+\ldots +\Lambda_{s_l}-\sum_{0\leq i\leq e-1} c^{e,{\bf s}}_i (\ulambda) \alpha_i,$$
 where $\Lambda_{\bf s}:=\Lambda_{s_1}+\ldots +\Lambda_{s_l}$ and 
$$ \Delta_{\bf s}:=\frac{1}{2} \sum_{1\leq i\leq l}\left(  (\frac{s_i^2}{e}-s_i) - ( \frac{s_i'^2}{e}-s_i')\right),$$
with $s_i'$ is the representant modulo $e$ of $s_i$ in $\{0,1,\ldots,e-1\}$. 
 Then we set:
 $$\| \ulambda \|^{(e,{\bf s})}:=\frac{ (\alpha^{e,{\bf s}} (\ulambda),\alpha^{e,{\bf s}} (\ulambda))}{2},\ \| \Lambda_{\bf s}\|=  \frac{ (\Lambda_{\bf s},\Lambda_{\bf s})}{2},$$ 
 so that 
  $$\begin{array}{rcl}
  \| \uemptyset \|^{(e,{\bf s})}  &=&\displaystyle \frac{1}{2}(-\Delta_{\bf s}\delta+\Lambda_{\bf s},-\Delta_{\bf s}\delta+\Lambda_{\bf s})\\
   &=& -\Delta_{\bf s}l +\| \Lambda_{\bf s}\|.
   \end{array}$$
  \begin{exa}
  For $l=1$ and ${ s}=0$ we have $\Delta_{ s}=0$ and $\| \Lambda_{ s}\|=0$ so that   $\| \uemptyset \|^{(e,0)}  =0$. 
  
  \end{exa}
  There is  an easy way to calculate $\| \ulambda \|^{(e,{\bf s})} $. The proof is in fact contained in \cite[Lemme 4.13]{Yvth} and is similar to \cite[Prop 8.1]{LLT}. We give it below for the convenience of the reader.  
  \begin{Prop}\label{b1} Let ${\bf s}\in \mathbb{Z}^l$ and 
  let $\ulambda\in \Pi^l$. Assume that $\umu\in \Pi^l$ is such that one can add an addable $i$-node to $\umu$ to obtain 
   $\ulambda$. Then we have 
   $$\| \umu \|^{(e,{\bf s})}  -\| \ulambda \|^{(e,{\bf s})}  =M_i ^{\bf s}(\umu)-1$$
  where $M_i^{\bf s} (\umu) $  is  the number of addable nodes of $\umu$ minus the number of removable nodes of $\umu$. 
  \end{Prop}
  \begin{proof}
  Under the above notation, we have that: 
    $$\displaystyle\begin{array}{rcl}
  \| \umu \|^{(e,{\bf s})} - \|\ulambda \|^{(e,{\bf s})}  &=&\displaystyle (1/2)\left( ( \alpha^{e,{\bf s}} (\umu) ,\alpha^{e,{\bf s}} (\umu) -( \alpha^{e,{\bf s}} (\umu) -\alpha_i,\alpha^{e,{\bf s}} (\umu)     -\alpha_i  \right)) \\
   &=&   (1/2)\left( 2( \alpha^{e,{\bf s}} (\umu) , \alpha_i) -(\alpha_i,\alpha_i)  \right) \\
   &=&   \displaystyle  (\alpha^{e,{\bf s}} (\umu) , \alpha_i) -1 \\
   \end{array}$$
   Now, by the previous definition of the weight $\alpha^{e,{\bf s}}$, we have $\alpha^{e,{\bf s}} (\umu) =\sum_{0\leq i\leq e-1} a_i \Lambda_i+d \delta$ 
   if and only if $\delta. \umu=d\umu$ and $t_i \umu=v^{a_i} \umu$. As by \cite[Th. 2.1]{Ug}, it is known that $t_i \umu=v^{M^{\bf s}_i (\umu)} \umu$, we can conclude. 
  \end{proof}
  
 It is now easy to compute the block weight $p_{(e,{\bf s})}$.
\begin{Prop}\label{b2}
Let ${\bf s}\in \mathbb{Z}^l$ and 
  let $\ulambda\in \Pi^l$.  
We have 
$$\| \ulambda \|^{(e,{\bf s})}  =\displaystyle    \| \uemptyset \|^{(e,{\bf s})}  -p_{(e,{\bf s})} (\ulambda) .$$

\end{Prop}
\begin{proof}
We can write:
$$\begin{array}{rcl}
(\alpha^{e,{\bf s}} (\ulambda),\alpha^{e,{\bf s}} (\ulambda)) &=& \displaystyle ( -\Delta_{\bf s} \delta +\Lambda_{\bf s}-\sum_{0\leq i\leq e-1} c^{e,{\bf s}}_i (\ulambda) \alpha_i,-\Delta_{\bf s} \delta +\Lambda_{\bf s}-\sum_{0\leq i\leq e-1} c^{e,{\bf s}}_i (\ulambda) \alpha_i)\\

&=&\displaystyle 2   \| \uemptyset \|^{(e,{\bf s})}  -2\sum_{0\leq i\leq e-1}  c^{e,{\bf s}}_i (\ulambda) (\Lambda_{\bf s}, \alpha_i) +\sum_{0\leq i,j\leq e-1} c^{e,{\bf s}}_i (\ulambda) c^{e,{\bf s}}_j (\ulambda)
 (\alpha_i,\alpha_j)\\
 &=& \displaystyle   2\| \uemptyset \|^{(e,{\bf s})}  -2\sum_{1\leq i\leq l}  c^{e,{\bf s}}_{s_i} (\ulambda)  +\sum_{0\leq i,j\leq e-1}(- c^{e,{\bf s}}_i (\ulambda) c^{e,{\bf s}}_{i-1} (\ulambda)
 +2 c^{e,{\bf s}}_i (\ulambda)^2 - c^{e,{\bf s}}_i (\ulambda) c^{e,{\bf s}}_{i+1} (\ulambda))\\
  &=&\displaystyle   2 \| \uemptyset \|^{(e,{\bf s})} -2\sum_{1\leq i\leq l}  c^{e,{\bf s}}_{s_i} (\ulambda)  +\sum_{0\leq i,j\leq e-1}(c^{e,{\bf s}}_i (\ulambda) - c^{e,{\bf s}}_{i-1} (\ulambda))^2\\
    &=&\displaystyle   2 \| \uemptyset \|^{(e,{\bf s})}  -2p_{(e,{\bf s})} (\ulambda) .  \\

\end{array}$$

\end{proof}
Combining these two propositions leads to:
  \begin{Prop}\label{b3}
 Let ${\bf s}\in \mathbb{Z}^l$  and  $\ulambda\in \Pi^l$. Assume that $\umu\in \Pi^l$ is such that one can add an addable $i$-node to $\umu$ to obtain 
   $\ulambda$. Then we have 
   $$p_{(e,{\bf s})} (\ulambda)  -p_{(e,{\bf s})} (\umu) =M_i ^{\bf s}(\umu)-1.$$
  \end{Prop}

The above proposition will be a crucial ingredient in the proof of one of our main results in the next section. 

\subsection{Computation of weights}

We here want to prove the following theorem. It mainly asserts that the block weight for an $l$-partition associated with a multicharge 
  can always been computed in terms of the usual block weight for a partition.  This result uses the map
   $\tau_{e,{\bf s}}$ defined in the previous section only for the multicharge  in $\overline{\mathcal{A}}^l_e$ (see Proposition \ref{emain}).

\begin{Th}\label{main}
 Let ${\bf s}\in\overline{\mathcal{A}}^l_e$  and  $\ulambda\in \Pi^l$ . We have:
 $$p_{(e,{\bf s})} (\ulambda)= p_{(e,s)} (\tau_{e,{\bf s}} (\ulambda))$$ where $s=\sum_{1\leq i\leq l} s_i$
\end{Th}
\begin{proof}
We argue by induction on the rank $n$ of $\ulambda$. Assume that $n=0$. Then $p_{(e,{\bf s})} (\ulambda)=0$ and by Proposition \ref{taucore}, $\tau_{e,{\bf s}} (\ulambda)$ is 
 an $e$-core so its weight is equal to $0$. Assume now that $n>0$. Let $\umu$ be an $l$-partition obtained from $\ulambda$ by deleting 
  a removable $i$-node for some $i\in \mathbb{Z}/e\mathbb{Z}$. By Proposition \ref{b3}, we get
$$p_{(e,{\bf s})} (\ulambda)  -p_{(e,{\bf s})} (\umu) =M_i ^{\bf s}(\umu)-1.$$
Now we have two cases to consider.
\begin{itemize}
\item Assume that $i\neq 0$ then $\tau_{e,{\bf s}} (\umu)$ is also obtained from 
$\tau_{e,{\bf s}} (\ulambda)$ by deleting a removable $i$-node. We have 
$$p_{(e,{\bf s})} (\tau_{e,{\bf s}} (\ulambda)) -p_{(e,{\bf s})}(\tau_{e,{\bf s}} (\umu)) =M_i ^{ s}(\tau_{e,{\bf s}} (\umu))-1.$$
We can thus conclude by induction using Proposition \ref{compa}
\item Assume that $i=0$ then to simplify the notation write $\mu:=\tau_{e,{\bf s}} (\umu)$ and 
$\lambda:=\tau_{e,{\bf s}} (\ulambda)$. We assume that the removable node corresponds to 
 a black bead of the $1$-abacus of $\tau_{e,{\bf s}} (\ulambda)$ in  position $x$.  
  By hypothesis, there is no black bead in position $x-(l-1)e-1$, because 
  $\tau_{e,{\bf s}} (\ulambda)$ and $\tau_{e,{\bf s}} (\umu)$ have the same abacus except that 
  the black beads between position $x$ and $x-(l-1)e-1$ are exchanged (and so are the empty position in the remaining one).  
Again, we will consider two cases:
\begin{itemize}
\item Assume that there is no black bead in position $x-(l-1)e$. Then one can consider $\nu$ the partition defined by the abacus obtained by moving the bead 
 in position $x$ from the abacus of $\lambda$  to the position $x-(l-1)e$. Its weight $p_{(e,{\bf s})} (\nu)$    is  equal to $p_{(e,{\bf s})} (\lambda) -(l-1)$ because  $\nu$ is obtained from $\lambda$ by  removing $l-1$ hooks from $\lambda$. 
  Now we have by Proposition \ref{b3}:
$$p_{(e,{\bf s})} (\nu)  -p_{(e,{\bf s})} (\mu) =M_0 ^{ s}(\mu)-1.$$
We conclude that 
$$p_{(e,{\bf s})} (\lambda)  -p_{(e,{\bf s})} (\mu) =M_0 ^{ s}(\mu)+l$$
that is, by Proposition \ref{compa}
$$p_{(e,{\bf s})} (\lambda)  -p_{(e,{\bf s})} (\mu) =M_0 ^{\bf s}(\umu)-1.$$
\item Assume that there is no black bead in position $x-(l-1)e$. Then we proceed in the opposite way: we define $\nu$ to be the partition obtained from $\lambda$  by moving
 the bead in position $x-(l-1)e$ to the position $x-(l-1)e$. Then $\mu$ is obtained from $\nu$ by moving the bead in position $x$ to the position $x-(l-1)e$ (which consists in removing $(l-1)$ $e$-hooks).
  We conclude exactly as in the previous case. 

\end{itemize}

\end{itemize}

\end{proof}
What can we do in the case where ${\bf s} \notin \overline{\mathcal{A}}^l_e$ ? In fact, one can use the procedure in \S \ref{esc} and associate to $\ulambda$ and ${\bf s}$ a multicharge $\widetilde{\bf s}^{\sigma_{\bf s}}\in \mathcal{A}^l_e$ and a multipartition $\ulambda^{\sigma_{\bf s}}$. It is clear from the definition that:
  $$p_{(e,{\bf s})} (\ulambda)= p_{(e,\widetilde{\bf s}^{\sigma_{\bf s}})} (\ulambda^{\sigma_{\bf s}})$$
which thus gives an effective way to compute the block weight in all cases.

\begin{Rem}
A (maybe more direct) proof might also be obtained using Proposition \ref{compare} but the above one has the advantage to avoid cumbersome computations.

\end{Rem}

\section{Further remarks  and applications}

In this section, we show how our main results simplify the block theory for Ariki-Koike algebras. In particular we show the relations of our work with some results by Fayers. By the definitions of cores and block weights, one can assume that ${\bf s} \in \mathcal{A}_e^l$ in this section. However, we will try to explain how all our results can be adapted to the general case 
${\bf s} \in \mathbb{Z}^l$.

\subsection{Cores of multipartitions}\label{expl}

Let us start with an easy corollary of Theorem \ref{main}.

\begin{Cor}
	Assume ${\bf s} \in \overline{\mathcal{A}}^l_e$. Then, the reduced $(e,{\bf s})$-core are exactly the elements of block weight $0$. 
\end{Cor}
\begin{proof}
	Let $\ulambda \in \Pi^l$, by Theorem \ref{main}, we have:
	$$p_{(e,{\bf s})} (\ulambda)= p_{(e,\sum_{1\leq i\leq l} s_i)} (\tau_{e,{\bf s}} (\ulambda))$$
	so $\ulambda$ is of block weight $0$ if and only if $\tau_{e,{\bf s}} (\ulambda)$ is of block weight $0$. Now, we know that the $e$-cores are exactly the partitions with block weight $0$ and we can thus conclude thanks to 
	Proposition \ref{emain}. 
\end{proof}

If we take ${\bf s}\in \mathbb{Z}^l$, then we have already noticed that:
$$p_{(e,{\bf s})} (\ulambda)= p_{(e,\widetilde{\bf s}^{\sigma_{\bf s}})} (\ulambda^{\sigma_{\bf s}}).$$
Since in addition  $\ulambda^{\sigma_{\bf s}}$ is a reduced  $\widetilde{\bf s}^{\sigma_{\bf s}}$-core if and only if 
$\ulambda$ is a $(e,{\bf s})$-core, we conclude that in the general case, the $(e,{\bf s})$-cores are exactly the elements of core weight $0$.

In \cite{Fa3}, Fayers has also introduced a notion of core for an $l$-partition associated with a multicharge.  His definition is the following one. Let  ${\bf s} \in \mathcal{A}_e^l$. Then an $l$-partition $\ulambda$ is a $(e| {\bf s})$-core if  there is no other $l$-partition $\umu$ such that 
$\mathcal{C}_{e,{\bf s}} (\ulambda)= \mathcal{C}_{e,{\bf s}} (\umu)$. 
In fact this coincides with our notion of $(e,{\bf s})$-cores. Indeed, by the results in \cite{Fa}, the  $(e| {\bf s})$-core  multipartitions are exactly the elements of weight $0$ (see \cite[Rem 2.3.1]{Fa3}) which are exactly the $(e,{\bf s})$-cores  by the above corollary.
In other words,  Definition \ref{escore1}  thus reveals the combinatorial structure of the $(e| {\bf s})$-cores introduced by Fayers.  Let  
us explain the consequences concerning the block theory of Ariki-Koike algebras and especially,  the similarities and the differences with the case $l=1$ that is, the case of the symmetric group. 

Let $\mathbb{F}\mathcal{H}_n^{{\bf s}}(\eta)$  be the Ariki-Koike algebra as defined in the introduction.  The representation theory of $\mathbb{F}\mathcal{H}_n^{{\bf s}}(\eta)$ 
is controlled by its decomposition matrix which we now briefly define. For all $l$-partition $\ulambda$; one can associate a certain finite dimensional 
$\mathbb{F}\mathcal{H}_n^{{\bf s}}(\eta)$-module $S^{\ulambda}$ called a Specht module. For each $M\in \operatorname{Irr} (\mathbb{F}\mathcal{H}_n^{{\bf s}}(\eta))$, we have the composition factor 
$[S^{\ulambda}:M]$. The matrix:
$$\mathcal{D}:=([S^{\ulambda}:M])_{\ulambda \in \Pi^{l}(n),M\in \operatorname{Irr} (\mathbb{F}\mathcal{H}_n^{{\bf s})}(\eta)}$$
is  the {\it decomposition matrix}.  By definition, two $l$-partitions $\ulambda$ and $\umu$ lie in the {\it same block} if 
there exists a sequence $(M_1,\ldots,M_r)$  of simple $\mathbb{F}\mathcal{H}_n^{{\bf s}}(\eta)$-modules and 
a sequence  of $l$-partitions $(\ulambda_1,\ldots,\ulambda_{r+1})$ with $\ulambda_1=\ulambda$, $\ulambda_{r+1}=\umu$ and 
for all $i\in \{1,\ldots,r\}$, we have $[S^{\ulambda_i}:M_i]\neq 0$ and 
$[S^{\ulambda_{i+1}}:M_i]\neq 0$. When $l=1$, we know that two partitions  are in the same block if and only if they have the same $e$-core and that their common weight 
is the number of $e$-hooks that can be removed to obtain this $e$-core. 
For $l>1$, a criterion has been provided by Lyle and Mathas \cite{LM} but it does not consist in any notion of hook or cores.  It asserts that $\ulambda$ and $\umu$ are in the same block of   $\mathbb{F}\mathcal{H}_n^{{\bf s}}(\eta)$ if  we have 
$\mathcal{C}_{e,{\bf s}} (\ulambda)= \mathcal{C}_{e,{\bf s}} (\umu)$.

Let $\ulambda$ be an $l$-partition of rank $n$.   To describe the blocks, one can restrict ourselves to the case   ${\bf s} \in \mathcal{A}_e^l$ (as usual the general case is derived by using the transformations in \S \ref{esc}).
We consider the  $(e,{\bf s})$-abacus $(L_{s_1} ,\ldots,L_{s_l} )$ of $\ulambda$. 
An {\it  elementary operation} on this abacus is defined as a move of one black bead from one runner of the abacus to another satisfying the following rule.
\begin{enumerate}
	\item If this black bead is not in the top runner, then we can do such an elementary operation on this black bead only if there is no black bead immediately above (that is in the same position on the runner just above). In this case, we slide 
	the black bead from its initial position, in a runner $i$,   to the runner $i+1$ located above in the same position.
	The resulting $l$-abacus corresponds to an $l$-partition  of rank $n-s_{i+1}+s_i-1$. 
	Indeed, when we add a black bead in the runner $i+1$ the rank becomes $n+ N-s_{i+1}-1$ for a certain integer $N$ and when we remove a   black bead from the runner $i$ in the same position, the rank becomes $n+ N-s_{i+1}-1 -(  N- s_i) $ that is 
	$n-s_{i+1}+s_i-1$.
	\item If this black bead is in the top runner in position $x$,  then we can do such an elementary operation only if there is no black bead in  position $x-e$ on the lowest runner. In this case, we slide the bead to the position $x-e$ of the lowest runner. 
	As above, the rank of the resulting $l$-partition is $n-(s_{1}-s_l+e+1)$
\end{enumerate}
Note that, after this procedure, the resulting multicharge associated with the $l$-abacus may not be in $\mathcal{A}_e^l$  but this  is not a problem: we can still perform  it  in the resulting abacus.  At the end, by construction, we obtain an $l$-abacus
$$(L_{v_1},\ldots,L_{v_l})$$ 
satisfying:
$$L_{v_1}\subset L_{v_2}\subset \ldots \subset L_l
\subset L_{v_{l} +e}.$$
This abacus is complete. Thus, by Proposition \ref{RedinAbar}, this corresponds to an $l$-partition $\umu$ and a multicharge ${\bf v}\in \overline{\mathcal{A}}^l_e$  such that $\umu$ is a reduced $(e,{\bf v})$-core.
 
\begin{Def} \label{Def_lcore}The  {\it core} of the $l$-partition $\ulambda$  associated with a multicharge ${\bf s}$ is the pair $(\umu,{\bf v})$ attached to $\ulambda$ and ${\bf s}$  by the previous procedure.
\end{Def}  

Doing an elementary operation on the  $(e,{\bf s})$-abacus  of $\ulambda$ as above  is equivalent to 
 remove one $e$-hook on the Young diagram of $\tau_{e,{\bf s}} (\ulambda)$. 
 As a consequence, 
by  Theorem \ref{main} and by the definition of the Uglov map,  $p_{(e,{\bf s})} (\ulambda)$ is the number of elementary operations we have made in this process to 
obtain our final abacus. The rank of the multipartition can also been computed thanks to the above remarks. The fact that this does not depend on the order in which the elementary operations are performed follows from the case $l=1$.

\begin{Rem}\label{Fam}
	\
	\begin{enumerate}	
	
	\item Take a black bead in the runner $i$ of the $l$-abacus  of $\ulambda$ in position $x$ such that there is no black bead in position $x-e$ in the same runner. Then one can always perform a series of $l$ elementary operations (as defined in the previous procedure) to obtain the same abacus except that the black bead in position $x$ moves to the position $x-e$ of the same runner.
	Indeed, let us denote by $b_1$ the bead in position $x$ in runner $i$, and consider all 
	 the beads $b_2$, \ldots, $b_k$ in position $x$ and runner $i_2$, \ldots, $i_k$ with $i<i_2\ldots<i_k$. 
	 Consider also the  beads $b_{k+1}$, \ldots, $b_r$ in position $x-e$ and runner $i_{k+1}$, \ldots, $i_r$ with $i_{k+1}\ldots<i_r<i$.
	 Then we can slide the bead $b_r$ in position $x-e$ in runner $i$, and then slide the bead $b_{r-1}$ to the position previously occupied by $b_r$ and so on. At the end, we obtain the desired abacus and we have made $l$ elementary operations to do that.  
	 In this case, the rank of the resulting $l$-partition is equal to $n-(s_{i+1}-s_{i}+1)-(s_{i+2}-(s_{i+1}+1)+1)-
	\cdots-(s_{1}-(s_{l}+1)+e+1)-\cdots-(s_{i}-1-(s_{i-1}+1)+1)$, that is $n-e$. 
	The $l$-partition so obtained is just the $l$-partition $\ulambda$ where a rim $e$-hook has been removed in  $\lambda^i$. This is thus consistent with our result. Nevertheless, this shorter hook removal procedure does not suffice to produce the core of $\ulambda$ for it can only yield a sequence of $l$ cores, that is a multicore.
	\item In \cite{Fa}, the notion of multicore is used instead of our notion of core.   From an arbitrary $l$-partition $\ulambda$, one can indeed associate another $l$-partition, with a smaller  block weight,  which may be seen as an "intermediate" between the   given $l$-partition and its $e$-core in the sense of Definition \ref{Def_lcore}. To do this, we can simply take the $e$-core of each partition or apply a sequence of elementary operations as we have just explained. We have already seen that the ${(e,\bf s})$-cores are multicores but the converse is not true in general. 
	\end{enumerate}
\end{Rem}

\begin{Cor}
	Two $l$-partitions with the same rank have the same  core  if and only if they belong to the same block of  $\mathbb{F}\mathcal{H}_n^{{\bf s}}(\eta)$. 
\end{Cor}
\begin{proof}
	This directly follows from Corollary \ref{block} together with the Lyle-Mathas characterization of blocks.  
	
\end{proof}
\begin{exa}
	Let us take ${\bf s}=(0,1,3)$ and $e=4$. We consider the two $3$-partitions $\ulambda=((3,2),(1,1),(2,2,1))$ and $\umu=((1),(4,2),(3,2))$ with Young diagrams:
	$$
	\left(
	\begin{array}{|c|c|c|}
	\hline
	0& 1  &2   \\
	\hline
	3 & 0 \\
	\cline{1-2}
	\end{array}\;,
	\begin{array}{|c|}
	\hline
	1  \\
	\hline
	0  \\
	\hline
	\end{array}\; ,
	\begin{array}{|c|c|}
	\hline
	3   &0     \\
	\hline
	2   & 3 \\
	\hline
	1 \\
	\cline{1-1}
	\end{array}
	\right),\qquad \qquad 
	\left(
	\begin{array}{|c|}
	\hline
	0 \\
	\hline
	\end{array}\;,
	\begin{array}{|c|c|c|c|}
	\hline
	1  & 2 & 3 & 0\\
	\hline
	0  & 1\\
	\cline{1-2}
	\end{array}\; ,
	\begin{array}{|c|c|c|}
	\hline
	3   &0  & 1    \\
	\hline
	2   & 3 \\
	\cline{1-2}
	\end{array}
	\right)$$
	They are in the same block because $\mathcal{C}_{e,{\bf s}} (\ulambda) =\mathcal{C}_{e,{\bf s}} (\umu) $. Now the $3$-abacus of $\ulambda$ is
	
	
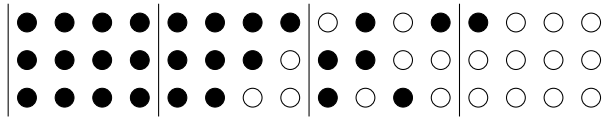
\begin{center}
		\begin{tikzpicture}[scale=0.5, bb/.style={draw,circle,fill,minimum size=2.5mm,inner sep=0pt,outer sep=0pt}, wb/.style={draw,circle,fill=white,minimum size=2.5mm,inner sep=0pt,outer sep=0pt}]

		\node [wb] at (9,2) {};
		\node [wb] at (8,2) {};
		\node [wb] at (7,2) {};
		\node [bb] at (6,2) {};
		\node [bb] at (5,2) {};
		\node [wb] at (4,2) {};
		\node [bb] at (3,2) {};
		\node [wb] at (2,2) {};
		\node [bb] at (1,2) {};
		\node [bb] at (0,2) {};
		\node [bb] at (-1,2) {};
		\node [bb] at (-2,2) {};
		\node [bb] at (-3,2) {};
		\node [bb] at (-4,2) {};
		\node [bb] at (-5,2) {};
		\node [bb] at (-6,2) {};

		\node [wb] at (9,1) {};
		\node [wb] at (8,1) {};
		\node [wb] at (7,1) {};
		\node [wb] at (6,1) {};
		\node [wb] at (5,1) {};
		\node [wb] at (4,1) {};
		\node [bb] at (3,1) {};
		\node [bb] at (2,1) {};
		\node [wb] at (1,1) {};
		\node [bb] at (0,1) {};
		\node [bb] at (-1,1) {};
		\node [bb] at (-2,1) {};
		\node [bb] at (-3,1) {};
		\node [bb] at (-4,1) {};
		\node [bb] at (-5,1) {};
		\node [bb] at (-6,1) {};

		\node [wb] at (9,0) {};
		\node [wb] at (8,0) {};
		\node [wb] at (7,0) {};
		\node [wb] at (6,0) {};
		\node [wb] at (5,0) {};
		\node [bb] at (4,0) {};
		\node [wb] at (3,0) {};
		\node [bb] at (2,0) {};
		\node [wb] at (1,0) {};
		\node [wb] at (0,0) {};
		\node [bb] at (-1,0) {};
		\node [bb] at (-2,0) {};
		\node [bb] at (-3,0) {};
		\node [bb] at (-4,0) {};
		\node [bb] at (-5,0) {};
		\node [bb] at (-6,0) {};

		\draw[](-6.5,-0.5)--node[]{}(-6.5,2.5);
		\draw[](-2.5,-0.5)--node[]{}(-2.5,2.5);
		\draw[](1.5,-0.5)--node[]{}(1.5,2.5);
		\draw[](5.5,-0.5)--node[]{}(5.5,2.5);
		\draw[](9.5,-0.5)--node[]{}(9.5,2.5);
		\end{tikzpicture}
	\end{center}
	
	To determine its core, we perform the above procedure and we obtain the following $3$-abacus:
	
	
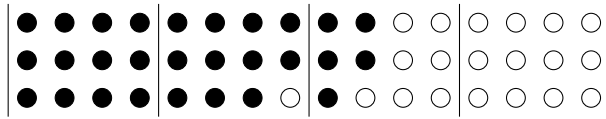
\begin{center}
		\begin{tikzpicture}[scale=0.5, bb/.style={draw,circle,fill,minimum size=2.5mm,inner sep=0pt,outer sep=0pt}, wb/.style={draw,circle,fill=white,minimum size=2.5mm,inner sep=0pt,outer sep=0pt}]

		\node [wb] at (9,2) {};
		\node [wb] at (8,2) {};
		\node [wb] at (7,2) {};
		\node [wb] at (6,2) {};
		\node [wb] at (5,2) {};
		\node [wb] at (4,2) {};
		\node [bb] at (3,2) {};
		\node [bb] at (2,2) {};
		\node [bb] at (1,2) {};
		\node [bb] at (0,2) {};
		\node [bb] at (-1,2) {};
		\node [bb] at (-2,2) {};
		\node [bb] at (-3,2) {};
		\node [bb] at (-4,2) {};
		\node [bb] at (-5,2) {};
		\node [bb] at (-6,2) {};

		\node [wb] at (9,1) {};
		\node [wb] at (8,1) {};
		\node [wb] at (7,1) {};
		\node [wb] at (6,1) {};
		\node [wb] at (5,1) {};
		\node [wb] at (4,1) {};
		\node [bb] at (3,1) {};
		\node [bb] at (2,1) {};
		\node [bb] at (1,1) {};
		\node [bb] at (0,1) {};
		\node [bb] at (-1,1) {};
		\node [bb] at (-2,1) {};
		\node [bb] at (-3,1) {};
		\node [bb] at (-4,1) {};
		\node [bb] at (-5,1) {};
		\node [bb] at (-6,1) {};

		\node [wb] at (9,0) {};
		\node [wb] at (8,0) {};
		\node [wb] at (7,0) {};
		\node [wb] at (6,0) {};
		\node [wb] at (5,0) {};
		\node [wb] at (4,0) {};
		\node [wb] at (3,0) {};
		\node [bb] at (2,0) {};
		\node [wb] at (1,0) {};
		\node [bb] at (0,0) {};
		\node [bb] at (-1,0) {};
		\node [bb] at (-2,0) {};
		\node [bb] at (-3,0) {};
		\node [bb] at (-4,0) {};
		\node [bb] at (-5,0) {};
		\node [bb] at (-6,0) {};

		\draw[](-6.5,-0.5)--node[]{}(-6.5,2.5);
		\draw[](-2.5,-0.5)--node[]{}(-2.5,2.5);
		\draw[](1.5,-0.5)--node[]{}(1.5,2.5);
		\draw[](5.5,-0.5)--node[]{}(5.5,2.5);
		\draw[](9.5,-0.5)--node[]{}(9.5,2.5);
		\end{tikzpicture}
	\end{center}
	the associated $(e,{\bf s})$-core is the $3$-partition $((1),\emptyset,\emptyset)$ together with the multicharge $(0,2,2)$ and the weight is $8$ because we perform $8$ moves of beads to obtain this 
	core.  Now if we consider $\umu$  whose $3$-abacus is
	
	
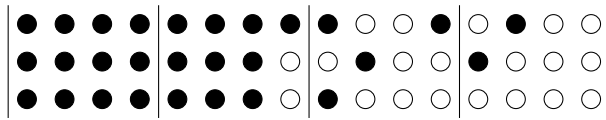
\begin{center}
		\begin{tikzpicture}[scale=0.5, bb/.style={draw,circle,fill,minimum size=2.5mm,inner sep=0pt,outer sep=0pt}, wb/.style={draw,circle,fill=white,minimum size=2.5mm,inner sep=0pt,outer sep=0pt}]

		\node [wb] at (9,2) {};
		\node [wb] at (8,2) {};
		\node [bb] at (7,2) {};
		\node [wb] at (6,2) {};
		\node [bb] at (5,2) {};
		\node [wb] at (4,2) {};
		\node [wb] at (3,2) {};
		\node [bb] at (2,2) {};
		\node [bb] at (1,2) {};
		\node [bb] at (0,2) {};
		\node [bb] at (-1,2) {};
		\node [bb] at (-2,2) {};
		\node [bb] at (-3,2) {};
		\node [bb] at (-4,2) {};
		\node [bb] at (-5,2) {};
		\node [bb] at (-6,2) {};

		\node [wb] at (9,1) {};
		\node [wb] at (8,1) {};
		\node [wb] at (7,1) {};
		\node [bb] at (6,1) {};
		\node [wb] at (5,1) {};
		\node [wb] at (4,1) {};
		\node [bb] at (3,1) {};
		\node [wb] at (2,1) {};
		\node [wb] at (1,1) {};
		\node [bb] at (0,1) {};
		\node [bb] at (-1,1) {};
		\node [bb] at (-2,1) {};
		\node [bb] at (-3,1) {};
		\node [bb] at (-4,1) {};
		\node [bb] at (-5,1) {};
		\node [bb] at (-6,1) {};

		\node [wb] at (9,0) {};
		\node [wb] at (8,0) {};
		\node [wb] at (7,0) {};
		\node [wb] at (6,0) {};
		\node [wb] at (5,0) {};
		\node [wb] at (4,0) {};
		\node [wb] at (3,0) {};
		\node [bb] at (2,0) {};
		\node [wb] at (1,0) {};
		\node [bb] at (0,0) {};
		\node [bb] at (-1,0) {};
		\node [bb] at (-2,0) {};
		\node [bb] at (-3,0) {};
		\node [bb] at (-4,0) {};
		\node [bb] at (-5,0) {};
		\node [bb] at (-6,0) {};

		\draw[](-6.5,-0.5)--node[]{}(-6.5,2.5);
		\draw[](-2.5,-0.5)--node[]{}(-2.5,2.5);
		\draw[](1.5,-0.5)--node[]{}(1.5,2.5);
		\draw[](5.5,-0.5)--node[]{}(5.5,2.5);
		\draw[](9.5,-0.5)--node[]{}(9.5,2.5);
		\end{tikzpicture}
	\end{center}
	and apply our procedure, one can check that we obtain the same core. 
	
\end{exa}
\begin{Rem}
	When $l$=1 and given an $e$-core $\lambda$, one can obtain directly all 
	the partitions in a fixed block with a given core weight $w$ by adding $w$ hooks  to $\lambda$ 
	while we stay in the set of partitions. 
	This process is less direct if $l>1$.  Let  $\ulambda$ be  a $(e,{\bf s})$-core.  We can assume that 
	${\bf s}\in \mathcal{A}^l_e$. Then if we perform 
	$w$ ``inverse'' elementary moves on its $l$-abacus, we obtain an $l$-partition $\umu$ 
	associated with a multicharge ${\bf s}'$ and the core of $\umu$  in   $\mathbb{F}\mathcal{H}_n^{{\bf s}'}(\eta)$   is 
	$(\ulambda,{\bf s})$. Now, still starting from the $e$-core,  if we  do $w$ other  ``inverse'' elementary moves on its $l$-abacus, one may obtain 
	another $l$-partition $\unu$ but also another multicharge ${\bf s}''$. Thus $\umu$ and $\unu$ will be in the same block of 
	$\mathbb{F}\mathcal{H}_n^{e,{\bf s}}(\eta)$ 
	if and only if ${\bf s}'={\bf s}''$. This means, one can obtain all the $l$-partitions in a fixed block of $\mathbb{F}\mathcal{H}_n^{{\bf s}'}(\eta)$  as in level $1$  except we have to keep only those with associated multicharge  $\bf s'$. 
	
\end{Rem}

\subsection{Multipartitions of small (block) weights}

As already noted in \cite{Fa}, in level $l>1$, each block of block weight $0$ contains exactly one $(e,{\bf s})$-core and thus is a simple block, as in the case $l=1$.  This  implies in particular that the Specht modules labeled by these $l$-partitions are irreducible
and that they coincide with their projective cover. 
 This shows that the  $(e,{\bf s})$-cores are always Uglov $l$-partitions. This is consistent with remark \ref{RT0}. 


%
%

In \cite[Th. 4.4]{Fa}, Fayers has given a description of the blocks of block weight $1$. Using our approach,  we here give an explicit characterization of these blocks.  When $l=1$, 
such blocks always contain exactly $e$ partitions. We will see that when $l>1$, this will depend on the multicharge we choose. 
Let ${\bf v}\in \mathbb{Z}^l$ and consider an $l$-partition $\umu$ with block weight $1$. The core af $\umu$ is the same as the core of 
the $l$-partition $\umu^{\sigma_{\bf v}}$ associated with the multicharge $\widetilde{\bf v}^{\sigma_{\bf v}}\in \mathcal{A}^l_e$. 
This means that we can in fact assume that    ${\bf v}\in \mathcal{A}_e^l$.

Now the  $l$-abacus of a  $l$-partition  $\umu$  with block weight $1$ for  the multicharge ${\bf v}$  can 
be  derived  from a reduced $(e,{\bf s})$-core $\ulambda$ where  ${\bf s}\in \overline{\mathcal{A}}^l_e$ by performing one   inverse elementary operation on the abacus of  $\ulambda$ (that is by inversing the procedure described in \S \ref{expl}). This  consists in moving  a black bead in position $x$   from a runner $i\in \{2,\ldots,l\}$ to the position $x$ of  the runner $i-1$, or 
from the runner $1$ in position $x$ to the runner $l$ in position $x+e$, if possible.

All the $l$-partitions $\umu$ of weight $1$ are then  obtained as follows:
\begin{itemize}
	\item For all $i\in \{1,\ldots,l-1\}$, if ${\bf s}:=(v_1,\ldots,v_{i-1}-1,v_{i}+1,\ldots,v_l)$ is such that ${\bf s}\in \overline{\mathcal{A}}_e^l$, they are obtained from a $(e,{\bf s})$-core $\ulambda$
	by doing one inverse elementary operation in its abacus from the runner $i$ to the runner $i-1$.  By definition 
	 of our notion of core, we can exactly do $v_i+1-(v_{i-1}-1)$ inverse elementary operations between the runner $i-1$ and the runner $i$. Thus, 
	we have 
	exactly $v_i-v_{i-1}+2$ multipartitions obtained from a given such  core  and they are all of the same rank
	$|\ulambda |+v_{i+1}-v_i+1$.
	\item If ${\bf s}:=(v_1+1,\ldots,v_l-1) $ is such that  ${\bf s}\in \overline{\mathcal{A}}_e^l$, they are 
	obtained from a $(e,{\bf s})$-core $\ulambda$ 
	by doing one inverse elementary operation in its abacus from the runner $1$ to the runner $l$.  We have 
	exactly $v_1-v_{l}+2+e$ multipartitions obtained from a given such core and they are all of the same rank     $|\ulambda |+v_{1}-v_l+e+1$.
\end{itemize}

\begin{Rem}
By \cite{LM}, the procedure described in this paper also gives the description of the blocks for affine Hecke algebras of type $A$. 
\end{Rem}

\begin{Rem}
	It is likely that the results of this paper may be used to study the block theory for the cyclotomic Hecke algebras of type $G(r,p,n)$.  Besides, 
	Theorem \ref{main}  gives  a correspondence between $(e,{\bf s})$-core  and $e$-cores which  could induce 
	similarities between blocks of Ariki-Koike algebras and blocks of Hecke algebras of type $A$. 
	We will come back to these questions in future works. 
\end{Rem}

\subsection{Examples}

We end this section with an example of computation of block weights and cores. We here take  $n=4$, $e=4$ and ${\bf s}=(0,1)$. Here is a table giving the block weight and the core of each $2$-partition. 
\begin{center}
	\begin{tabular}{ | c | c | c |}
		\hline 
		$2$-partition & core & block weight \\
		\hline			
		$((4),\emptyset)$ & $(\uemptyset ;(0,1) )$ & 2 \\
		\hline
		$((3),(1))$ & $((\emptyset,(1,1)) ;(0,3) )$ & 1 \\
		\hline
		$(\emptyset,4)$ & $(\uemptyset ;(0,1) )$   &2\\
		\hline  
		$((3,1),\emptyset)$ & $(\uemptyset ;(0,1) )$   &2\\
		\hline  
		$((2),(2))$ & $((\emptyset,1.1) ;(0,3) )$ & 1 \\
		\hline  
		$((1),(3))$ & $(\uemptyset ;(0,1) )$   &2\\
		\hline
		$((2,2),\emptyset)$ & $(((2),\emptyset) ;(0,3) )$   &1\\
		\hline
		$((2,1),(1))$  & $(((2,1),1) ;(0,1) )$   &0\\
		\hline
		$((2,1,1),\emptyset)$ & $(\uemptyset ;(0,1) )$   &2\\
		\hline
		$((2),(1,1))$ & $( ((2),(1,1)) ;(0,1) )$   &0\\
		\hline
	\end{tabular}	
		\begin{tabular}{ | c | c | c |}
		\hline 
		$2$-partition & core & block weight \\
		\hline
		$((1,1),(2))$ & $(\uemptyset ;(0,1) )$   &2\\
		\hline
		$((1),(2,1))$ & $( ((1),(2,1)) ;(0,1) )$   &0\\
		\hline
		$((1,1),(1,1))$ & $(((2),\emptyset) ;(0,3) )$   &1\\   
		\hline
		$(\emptyset,(3,1))$ & $(\uemptyset ;(0,1) )$   &2\\
		\hline
		$((1,1,1),(1))$& $(\uemptyset ;(0,1) )$   &2\\  
		\hline  
		$(\emptyset,(2,2))$ & $((\emptyset,(1,1)) ;(0,3) )$ & 1 \\ 
		\hline
		$((1,1,1,1),\emptyset)$ & $(\uemptyset ;(0,1) )$   &2\\  
		\hline
		$(\emptyset,(2,1,1))$ & $(\uemptyset ;(0,1) )$   &2\\
		\hline
		$((1),(1,1,1))$ & $(((2),\emptyset) ;(0,3) )$   &1\\      
		\hline 
		$(\emptyset,(1,1,1,1))$ & $(\uemptyset ;(0,1) )$   &2\\
		\hline	
	\end{tabular}
	
\end{center}

Note that the core of the blocks of block weight $1$ are always associated with the same multicharge, which is ${\bf v}=(0,3)$ and there is two different cores which gives $3=v_2-v_1$ elements in the same block in both cases.  
The rank of this core is the  $n-(s_2-s_1+1)=2$. 
This is consistent with the results of the previous section. 

\bigskip
Let us consider now the multicharge ${\bf s}:=(0,1,3)$ with $e=4$ and $n=4$. Then
\begin{itemize}
	\item we have seven $(e,{\bf s})$-cores: $(\emptyset,(3,1)$, $(\emptyset,(1),(1,1,1))$,
	$((1),(2,1),\emptyset)$, $((2),(1,1),\emptyset)$, $(\emptyset,(2),(1,1))$, $((1,1),\emptyset,(2))$, $((1),\emptyset,(2,1))$.
	\item We have three blocks of clock weight $1$ which are:
	\begin{itemize}
		\item $\{((2,1),\emptyset,(1)), ((1,1),(1),(1)),(\emptyset,(1,1,1),(1))\}$ with $(((1),\emptyset,(1)),(-1,2,3))$ as a core. 
		\item $\{((3),(1),\emptyset), ((1,1),(1),(1)),(\emptyset,(2,2),\emptyset)\}$ with $(\emptyset,(1,1),\emptyset),(-1,2,3))$ as a core. 
		\item $\{(\emptyset,(1),(3)), (2,1,1),((2,1),(1),\emptyset)\}$ with $((1),(1),\emptyset),(1,1,2))$ as a core.   
	\end{itemize}
	
\end{itemize}
Again, this is consistent with the results of the previous section. 


\end{document}